\documentclass{article}
\usepackage[english]{babel}
\usepackage{amsfonts}
\usepackage{amsmath,amsthm,amssymb,amsrefs}
\usepackage{color}
\usepackage{authblk}
\usepackage{graphicx}
\usepackage{tikz}
\usepackage{stmaryrd}
\usepackage{mathrsfs,upgreek}
\usepackage{eucal}
\usepackage{enumerate}
\usepackage{changes}
\usepackage[left=3cm,right=3cm,top=2.5cm,bottom=2.5cm]{geometry}

\newtheorem{theorem}{Theorem}[section]
\newtheorem{lemma}[theorem]{Lemma}
\newtheorem{proposition}[theorem]{Proposition}
\newtheorem{corollary}[theorem]{Corollary}
\theoremstyle{definition}

\newtheorem{example}[theorem]{Example}

\newcommand{\op}[1]{\textrm{\upshape #1}}

\newcommand{\join}{\vee}

\newcommand{\meet}{\wedge}

\newcommand{\la}{\langle}
\newcommand{\ra}{\rangle}
\newcommand{\alg}[1]{{\textbf{\upshape #1}}}  %
\newcommand{\vv}[1]{\mathsf {#1}}

\newcommand{\f}{\varphi}

\renewcommand{\th}{\theta}

\newcommand{\sse}{\subseteq}

\newcommand{\app}{\approx}
\newcommand{\HH}{{\mathbf H}}  % homomorhic images
\newcommand{\II}{{\mathbf I}} % isomorphic copies
\newcommand{\SU}{{\mathbf S}} % subalgebras
\newcommand{\PP}{{\mathbf P}}   % direct products
\newcommand{\VV}{{\mathbf V}}   % variety gen. by ...

\newcommand{\MTL}{\vv M\vv T\vv L}

\newcommand{\GMTL}{\vv G\vv M\vv T\vv L}

\newcommand{\BL}{\vv B\vv L}
\newcommand{\BH}{\vv B\vv H}

\newcommand{\ib}{\item[$\bullet$]}

\newcommand{\vuc}[2]{#1_1,\dots,#1_{#2}}

\newcommand{\imp}{\rightarrow}

\begin{document}
\title{Projectivity in (bounded) commutative integral\\ residuated lattices}
\author{Paolo Aglian\`{o}\\
DIISM\\
Universit\`a di Siena, Italy\\
agliano@live.com
\and
 Sara Ugolini\\
 Artificial Intelligence Research institute (IIIA), CSIC\\
 Bellaterra, Barcelona, Spain\\
 sara@iiia.csic.es}
\date{}
\maketitle

\begin{abstract}
In this paper we study projective algebras in varieties of (bounded) commutative integral residuated lattices. We make use of a well-established construction in residuated lattices, the ordinal sum, and the order property of divisibility. Via the connection between projective and splitting algebras, we show that the only finite projective algebra in $\mathsf{{FL}_{ew}}$ is the two-element Boolean algebra. Moreover, we show that several interesting varieties have the property that every finitely presented algebra is projective, such as locally finite varieties of hoops. Furthermore, we show characterization results for finite projective Heyting algebras, and finitely generated projective algebras in locally finite varieties of bounded hoops and BL-algebras.
Finally, we connect our results with the algebraic theory of unification.
\end{abstract}

%%%%%%%%%%%%%%%%%%%%%%%%%%%%%%%%%%%%%%%%%%%%%%%%%%%%%%%%%%%%%%%%%%%%%%
%% MAIN MATTER
%%%%%%%%%%%%%%%%%%%%%%%%%%%%%%%%%%%%%%%%%%%%%%%%%%%%%%%%%%%%%%%%%%%%%%

\section{Introduction}\label{sec:intro}
In this paper we approach the study of projective algebras in varieties of (bounded) commutative integral residuated lattices. Our point of view is going to be  algebraic, however the reader may keep in mind that projectivity is a categorical concept, and therefore our findings pertain the corresponding algebraic categories as well. Being projective in a variety of algebras, or in any class containing all of its free objects, corresponds to being a retract of a free algebra, and projective algebras contain relevant information both on their variety and on its lattice of subvarieties.

In particular, as first noticed by McKenzie \cite{McKenzie1972}, there is a close connection between projective algebras in a variety and {\em splitting algebras} in its lattice of subvarieties. The notion of splitting comes from lattice theory, and studying splitting algebras is particularly useful to understand lattices of subvarieties. Essentially, a splitting algebra generates a variety that divides the subvariety lattice in the disjoint union of its principal filter and a principal ideal generated by a variety axiomatizable by a single identity (called {\em splitting equation}). In particular, we shall see that subdirectly irreducible projective algebras are splitting for the variety they belong to.

The varieties of algebras which are object of our study are relevant both in the realm of algebraic logic and from a purely algebraic point of view. In fact, residuated structures arise naturally in the study of many interesting algebraic systems, such as ideals of rings \cite{WardDilworth39} or lattice-ordered groups \cite{BCGJT}, besides encompassing the equivalent algebraic semantics (in the sense of Blok-Pigozzi \cite{BlokPigozzi1989}) of so-called substructural logics. We refer the reader to \cite{GJKO} for detailed information on this topic. The Blok-Pigozzi notion of algebraizability entails that the logical deducibility relation is fully and faithfully represented by the algebraic equational consequence of the corresponding algebraic semantics, and therefore logical properties can be studied algebraically, and viceversa. Substructural logics are a large framework and include many of the interesting non-classical logics: intuitionistic logics, relevance logics, and fuzzy logics to name a few, besides including classical logic as a special case.
Therefore, substructural logics on one side, and residuated lattices on the other, constitute a wide unifying framework in which very different structures can be studied uniformly.

The investigation of projective structures in particular varieties of residuated lattices has been approached by several authors (see for instance \cite{BalbesHorn1970,CabrerMundici2009,DantonaMarra2006,DiNolaGrigoliaLettieri2008,Ghilardi1997,Ugolini2022}). However, to the best of our knowledge, no effort has yet been done to provide a uniform approach in a wider framework, which is what we attempt to start in this manuscript.

In the general setting of $\mathsf{FL_{ew}}$-algebras, which are bounded commutative integral residuated lattices, via the connection with splitting algebras, we show that the only finite projective algebra in $\mathsf{FL_{ew}}$ is the two-element Boolean algebra $\alg 2$. Besides some general findings on $\mathsf{FL_{ew}}$-algebras, our other main results will concern varieties where the lattice order is  the inverse divisibility ordering, that is, where $x \leq y$ if and only if there exists $z$ such that $x = z \cdot y$. We show that several interesting varieties in the realm of algebraic logic have the property that every finitely presented algebra is projective, such as, in particular, all locally finite varieties of hoops and their implicative reducts. Moreover, as a consequence of some general results about projectivity in varieties closed under the construction of ordinal sums, we show an alternative proof of the characterization of finite projective Heyting algebras shown in \cite{BalbesHorn1970}. 

Interestingly, the study of projective algebras in this realm has a relevant logical application. Indeed, following the work of Ghilardi \cite{Ghilardi1997}, the study of projective algebras in a variety is strictly related to unification problems for the corresponding logic. More precisely, the author shows that a unification problem can be interpreted as a finitely presented algebra, and solving such a problem means finding an homomorphism to a projective algebra. We will explore some consequences of our study in this context in Section \ref{section:unif}.

We believe that our results, despite touching several relevant subvarieties of $\mathsf{FL_{ew}}$, have only scratched the surface of the study of projectivity in this large class of algebras, and shall serve as ground and inspiration for future work.

\section{Preliminaries}
Given a class $\vv K$ of algebras, an algebra $\alg A \in \vv K$ is {\em projective} in $\vv K$ if for all $\alg B,\alg C \in \vv K$, any homomorphism $h\colon \alg A \longmapsto \alg C$, and any surjective homomorphism $g\colon \alg B\longmapsto \alg C$, there is a homomorphism $f\colon \alg A \longmapsto \alg B$ such that $h=gf$.

\begin{figure}[htbp]
\begin{center}
\begin{picture}(200,70)
\put(60,60){\makebox(0,0){$\alg A$}}
\put(135,60){\makebox(0,0){$\alg C$}}
\put(135,10){\makebox(0,0){$\alg B$}}
\put(70,60){\vector(1,0){55}}
\put(135,20){\vector(0,1){30}}
\put(135,20){\vector(0,1){25}}
\put(65,53){\vector(3,-2){60}}
\put(91,22){\makebox(0,0){\footnotesize$f$}}
\put(145,32){\makebox(0,0){\footnotesize$g$}}
\put(98,68){\makebox(0,0){\footnotesize$h$}}
\end{picture}
\end{center}
\end{figure}
\vspace{-.15 in}

Determining the projective algebras in a class is usually a challenging problem, especially in a general setting. If however $\vv K$ contains all the free algebras on $\vv K$ (in particular, if $\vv K$ is a variety of algebras), projectivity admits a simpler formulation. We call an algebra $\alg B$ a {\em retract} of an algebra $\alg A$ if there is a homomorphism $g\colon \alg A \longmapsto \alg B$ and a homomorphism $f\colon\alg B \longmapsto \alg A$ with $gf= \op{id}_\alg B$ (and thus, necessarily, $f$ is injective and $g$ is surjective). The following theorem was proved first by Whitman for lattices \cite{Whitman1941} but it is well-known to hold for any class of algebras.

\begin{theorem}\label{whitman} Let $\vv K$ be a class of algebras containing all the free algebras in $\vv K$ and let $\alg A \in \vv K$. Then the following are equivalent:
\begin{enumerate}
\item $\alg A$ is projective in $\vv K$.
\item $\alg A$ is a retract of  a free algebra in $\vv K$.
\item $\alg A$ is a retract of a projective algebra in $\vv K$.
\end{enumerate}
In particular every free algebra in $\vv K$ is projective in $\vv K$.
\end{theorem}

It follows immediately that if $\vv V$ is a variety and $\alg A$ is projective in $\vv V$, then it is projective for any subvariety $\vv W$ of $\vv V$ to which it belongs; equivalently, if $\alg A \in \vv W$ is not projective in $\vv W$ then it cannot be projective in any supervariety of $\vv W$.

Let us also observe that the algebraic definition of projectivity we have given corresponds to a categorical notion for the corresponding algebraic category (where the objects are the algebras in $\vv K$ and whose morphisms are the homomorphisms between algebras in $\vv K$), where surjective homomorphisms are regular epimorphisms.  It therefore follows that projectivity is preserved by categorical equivalence:
the projective objects in one category are exactly the images, through the  functor yielding the equivalence, of the projective objects in the other category.

In general, determining all projective algebras in a variety is a complicated task. As examples of success in characterizing projective algebras in interesting classes we can quote:
\begin{enumerate}
\ib A finite lattice is projective in the variety of all lattices if and only if it semidistributive and satisfies Whitman's condition (W) (see \cite{Nation1982} for details).
\ib A finite distributive lattice is projective in the variety of distributive lattices if and only if the meet of any two meet irreducible elements is again meet irreducible \cite{Balbes1967}.
\ib A Boolean algebra is projective in the variety of distributive lattices if and only if it is finite \cite{Balbes1967}; hence every finite Boolean algebra is projective in the variety of Boolean algebras.
\ib An abelian group is projective in the variety of abelian groups if and only if it is free.
\ib If $\vv M_\alg D$ is the variety of left modules on a principal ideal domain $\alg D$, then a module is projective in $\vv M_\alg D$ if and only if it is free (this is a consequence of the so called {\em Quillen-Suslin Theorem}, see \cite{Quillen1976}).  
\ib Even more generally \cite{Quackenbush1971}, if $\alg A$ is a finite quasi-primal algebra, then if $\alg A$ has no minimal nontrivial subalgebra, each finite algebra in $\VV(\alg A)$ is projective;
alternatively a finite algebra in $\VV(\alg A)$ is projective if and only if it admits each minimal subalgebra of $\alg A$ as a direct decomposition factor.
\end{enumerate}
In what follows, we will be mainly interested in projective algebras that are finitely generated. We will see that in many interesting cases, these will correspond to \emph{finitely presented} algebras in the variety.
An algebra is finitely presented if it can be defined by a finite number of generators and satisfies finitely many identities.
For any set $X$ and any variety $\vv V$ we will denote by $\alg F_\vv V (X)$ the free algebra in $\vv V$ over $X$. Thus, we say that an algebra $\alg A\in \vv V$ is {\em finitely presented} in $\vv V$ if there is a finite set $X$ and a compact congruence $\th \in \op{Con}(\alg F_\vv V (X))$ such that $\alg F_\vv V(X)/\th \cong \alg A$.

Note that if $\vv V$ has finite type, every finite algebra in $\vv V$ is finitely presented, and if $\vv V$ is locally finite every finitely presented algebra in $\vv V$ is finite. We remark that the notion of finitely presented algebra is a categorical notion \cite{GabrielUllmer1971}, and thus it is preserved under categorical equivalence.
The proof of the following theorem is standard (the reader can see \cite{Ghilardi1997}).

\begin{theorem}\label{finitelypresented} For a finitely presented algebra $\alg A \in \vv V$ the following are equivalent:
\begin{enumerate}
\item $\alg A$ is projective in $\vv V$;
\item $\alg A$ is projective in the class of all finitely presented algebras in $\vv V$;
\item $\alg A$ is a retract of a finitely generated free algebra in $\vv V$.
\end{enumerate}
\end{theorem}

\subsection{(Bounded) commutative integral residuated lattices}
In the sequel we will deal mainly with (bounded) commutative and integral residuated lattices.

A commutative integral residuated lattice (or CIRL) is an algebra\\ $\la A,\join,\meet, \cdot, \imp,1\ra$ such that
\begin{enumerate}
\item $\la A,\join,\meet, 1 \ra$ is a lattice with largest element $1$;
\item $\la A,\cdot,1\ra$ is a commutative monoid;
\item $(\cdot,\imp)$ form a residuated pair w.r.t. the lattice ordering, i.e. for all $a,b,c \in A$
$$
a \cdot b \le c\qquad\text{if and only if}\qquad a \le b \imp c.
$$
\end{enumerate}
In what follows, we will often write $xy$ for $x \cdot y$. We call {\em bounded commutative integral residuated lattice} a CIRL with an extra constant $0$ in the signature that is the least element in the lattice order.

(Bounded) commutative and integral residuated lattices form varieties called ($\mathsf{FL_{ew}}$) $\mathsf{CIRL}$. Some of the results of this paper will concern commutative and integral residuated (meet-){\em semilattices}, that is, the $\lor$-free subreducts of commutative and integral residuated lattices.

Residuated lattices are rich structures, and many equations hold in them; for an equational axiomatization and a list of valid identities we refer the reader to \cite{BlountTsinakis2003}. We focus on two equations that bear interesting consequences, i.e., prelinearity and divisibility:
\begin{align}
&(x \imp y) \join (y \imp x) \app 1.\tag{prel}\\
&x(x \imp y) \app y(y \imp x); \tag{div}
\end{align}
It can be shown (see \cite{BlountTsinakis2003} and \cite{JipsenTsinakis2002}) that a subvariety of $\mathsf{FL_{ew}}$ satisfies the prelinearity equation (prel) if and only if any algebra therein is a subdirect product of totally ordered algebras, and this implies via Birkhoff's Theorem that all the subdirectly irreducible algebras are totally ordered. Such varieties are called {\em representable} (or {\em semilinear}) and the subvariety axiomatized by (prel) is the largest subvariety of $\mathsf{FL_{ew}}$ that is representable; such variety is usually denoted by $\MTL$, since it is the equivalent algebraic semantics of Esteva-Godo's {\em Monoidal t-norm  based logic} \cite{EstevaGodo2001}.

If a variety satisfies the divisibility condition (div) then the lattice ordering becomes the inverse divisibility ordering: for any algebra $\alg A$ therein and for all $a,b \in A$:
$$
a \le b \qquad\text{if and only if}\qquad \text{there is $ c \in A$ with $a =bc$}.
$$
In this case, it can be shown that the meet is definable as $a\meet b = a(a \imp b)$.
We will call $\mathsf{FL_{ew}}$-algebras satisfying (div) {\em hoop algebras} or {\em $\mathsf{HL}$-algebras}, and we denote the variety they form by $\mathsf{HL}$.

If an algebra in $\mathsf{FL_{ew}}$ satisfies both (prel) and (div) then it is called a $\BL$-algebra and the variety of all $\BL$-algebras is denoted by $\BL$. Again the name comes from logic: the variety of $\BL$-algebras is the equivalent algebraic semantics of {\em H\'ajek's Basic Logic} \cite{Ha98}. A systematic investigation of varieties of $\BL$-algebras started with \cite{AglianoMontagna2003} and it is still ongoing (see \cite{Agliano2017c} and the bibliography therein).

It follows from the definition that given a variety of bounded commutative integral residuated lattices, the class of its {\em $0$-free subreducts} is a class of residuated lattices; we have a very general result.

\begin{lemma} Let $\vv V$ be any subvariety of $\mathsf{FL_{ew}}$; then the class $\SU^0(\vv V)$ of the zero-free subreducts of algebras in $\vv V$ is a variety. Moreover if $\alg A \in \SU^0(\vv V)$ is bounded (i.e. there is a minimum in the ordering), then it is polynomially equivalent to an algebra in $\vv V$.
\end{lemma}
\begin{proof}The proof of the first claim is in Proposition 1.10 of \cite{AFM}; it is stated for varieties of $\BL$-algebras but it uses only the description of the congruence filters, that can be used in any subvariety of $\mathsf{FL_{ew}}$ (as the reader can easily check). The second claim is almost trivial and the proof is left to the reader.
\end{proof}
In particular, we point out the following relevant examples of varieties of zero-free subreducts:
\begin{enumerate}
\ib $\SU^0(\mathsf{FL_{ew}}) = \mathsf{CIRL}$.
\ib $\SU^0(\MTL) = \mathsf{CIRRL}$, the variety of commutative integral representable residuated lattices, sometimes also called $\GMTL$.
\ib $\SU^0(\mathsf{HL})$ is the variety of commutative and integral residuated lattices satisfying divisibility (corresponding to commutative and integral $\mathsf{GBL}$-algebras). Notice that these algebras have been called {\em hoops} in several papers, but the original definition of hoop is different (see \cite{BlokFerr2000}): a hoop is the variety of $\{\meet,\cdot,1\}$-subreducts of $\mathsf{HL}$-algebras (no join!). It is not true that all hoops have a lattice reduct, and
the ones that do are exactly the $\lor$-free reducts of commutative and integral $\mathsf{GBL}$-algebras, and we will refer to them as \emph{full hoops}.
\ib $\SU^0(\BL) = \BH$, the variety of {\em basic hoops} \cite{AFM}.
Basic hoops are full hoops, since
the prelinearity equation makes the join definable using $\meet$ and $\imp$ (see for instance \cite{Agliano2018c}):
\begin{equation}
((x \imp y) \imp y) \meet ((y \imp x) \imp x) \app x \join y. \tag{PJ}
\end{equation}
\end{enumerate}

With respect to the structure theory, congruences of algebras in $\mathsf{FL_{ew}}$ are very well behaved; as a matter of fact any subvariety of $\mathsf{FL_{ew}}$ is {\em ideal determined} w.r.t $1$, in the sense of \cite{GummUrsini1984} (but see also \cite{OSV2}). In particular this implies that congruences are totally determined by their $1$-blocks. If $\alg A \in \mathsf{FL_{ew}}$ the $1$-block of a congruence of $\alg A$ is called a {\em deductive filter} (or {\em filter} for short); it can be shown that a filter of $\alg A$ is an order filter containing $1$ and closed under multiplication. Filters form an algebraic lattice isomorphic with the congruence lattice of $\alg A$ and if $X \sse A$ then the filter generated by $X$ is
$$
\op{Fil}_\alg A(X) = \{a \in A: x_1\cdot \ldots \cdot x_n \le a, \text{for some $n \in \mathbb N$ and $\vuc xn \in X$}\}.
$$
The isomorphism between the filter lattice and the congruence lattice is given by the maps:
\begin{align*}
& \th \longmapsto 1/\th\\
& F \longmapsto \th_F = \{(a,b): a \imp b,\ b \imp a \in F\},
\end{align*}
where $\th$ is a congruence and $F$ a filter.

\subsection{Ordinal sums}\label{section:ordinalsums}
A powerful  tool for investigating (bounded) integral residuated lattices in general, and $\MTL$-algebras in particular, is the {\em ordinal sum} construction. Let $\alg A_0,\alg A_1 \in \mathsf{CIRL}$ such that $\alg A_0 \cap \alg A_1 = \{1\}$, and consider $A_0 \cup A_1$. The ordering intuitively stacks $A_{1}$ on top of $A_{0} \setminus \{1\}$ and more precisely it is given by
$$
a \le b \quad\text{if and only if} \quad \left\{
             \begin{array}{l}
               \hbox{$b=1$, or} \\
               \hbox{$a \in A_0\setminus\{1\}$ and $b \in A_1\setminus\{1\}$ or} \\
               \hbox{$a,b \in A_i\setminus\{1\}$ and $a \le_{A_i} b$, $i=0,1$.}
             \end{array}
           \right.
$$
Moreover, we define the product inside of the two components to be the original one, and between the two different components to be the meet, and consequently we have the following:
\begin{align*}
&a\cdot b = \left\{
              \begin{array}{ll}
                a, & \hbox{if $a \in A_0\setminus\{1\}$ and $b \in A_1$;} \\
                b, & \hbox{if $a \in A_1$ and $b\in A_0\setminus\{1\}$;}\\
                a  \cdot_{A_i} b, &\hbox{if $a,b \in A_i$, $i=0,1$.}
              \end{array}
            \right.\\
&a \imp b = \left\{
              \begin{array}{ll}
                b, &\hbox{if $a \in A_1$ and $b\in A_0\setminus\{1\}$;} \\
                1, &\hbox{if $a \in A_0\setminus\{1\}$ and $b \in A_1$;} \\
                a \imp_{A_i} b, &\hbox{if $a,b \in A_i$, $i=0,1$.}
              \end{array}
            \right.
\end{align*}
If we call $\alg A_0 \oplus \alg A_1$ the resulting structure, it is easily checked that $\alg A_0 \oplus \alg A_1$ is a semilattice ordered integral and commutative residuated monoid. However, it might not be a residuated lattice, and the reason is that if $1_{A_0}$ is not join irreducible, and $\alg A_1$ is not bounded, we run into trouble. In fact, if $a,b \in A_0\setminus\{1\}$ and $a \join_{A_0} b =1_{A_0}$ then  the upper bounds of $\{a,b\}$ all lie in $A_1$; and since $A_1$ is not bounded there is no least upper bound of $\{a,b\}$ in $\alg A_0 \oplus \alg A_1$, thus the ordering is not a lattice ordering.
However, if $1_{A_0}$ is join irreducible, then the problem disappears, and we can define $\alg A_0 \oplus \alg A_1$ as before. While if $1_{A_0}$ is not join irreducible but $\alg A_1$ is bounded, say by $u$, then we can define:
$$
a \join b = \left\{
              \begin{array}{ll}
                a, & \hbox{$a \in A_1$ and $b \in A_0$;} \\
                b, & \hbox{$a \in A_0$ and $b \in A_1$;} \\
                a \join_{A_1} b, & \hbox{if $a,b \in A_1$;} \\
                a \join_{A_0} b, & \hbox{if $a,b \in A_0$ and $a \join_{A_0} b < 1$;}\\
                u, & \hbox{if $a,b \in A_0$ and $a \join_{A_0} b = 1$.}\\
              \end{array}
            \right.
$$
We will therefore call $\alg A_0 \oplus \alg A_1$ the {\em ordinal sum} of $\alg A_0$ and $\alg A_1$, and we will say that the ordinal sum {\em exists} if $\alg A_0 \oplus \alg A_1 \in \mathsf{CIRL}$. If $\alg A_0 \in \mathsf{FL_{ew}}$ and the ordinal sum of $\alg A_0$ and $\alg A_1$ exists, then $\alg A_0 \oplus \alg A_1 \in \mathsf{FL_{ew}}$.

Every time we deal with a class of $\mathsf{CIRL}$ for which we know that the ordinal sum always exists, we can define the ordinal sum of a (possibly infinite) family of algebras in that class; in that case the family is indexed by a totally ordered set $\la I,\le\ra$ that may or may not have a minimum (but it does in case we are in $\mathsf{FL_{ew}}$).

For an extensive treatment of ordinal sums, even in more general cases, we direct the reader to \cite{Agliano2018a,Galatos2005,GJKO}.
Here we would like to point out some facts that will be useful in the following sections.

We call an algebra in $\mathsf{CIRL}$ {\em sum irreducible} if it cannot be written as the ordinal sum of at least two nontrivial algebras in $\mathsf{CIRL}$.
Then, by a straightforward application of Zorn's lemma (see for instance Theorem 3.2 in \cite{Agliano2018b}), we obtain the following.
\begin{proposition}
	Every algebra in $\mathsf{CIRL}$ is the ordinal sum of sum irreducible algebras in $\mathsf{CIRL}$.
\end{proposition}
In general, we do not know what the sum irreducible algebras in a subvariety of $\mathsf{CIRL}$ may be. A most recognized result is the classification of all totally ordered sum irreducible $\BL$-algebras and basic hoops in terms of Wajsberg hoops \cite{AglianoMontagna2003}.

It is possible to describe the behavior of the classical operators $\HH,\PP,\II,\PP_u$ (denoting respectively homomorphic images, direct products, isomorphic images and ultraproducts) on ordinal sums; the reader can consult Section 3 of \cite{AglianoMontagna2003} and  Lemma 3.1 in \cite{Agliano2018a}; subalgebras in the bounded case are slightly more critical since there is the constant $0$ which must be treated carefully. If we have a family $(\alg A_i)_{i \in  I}$ of algebras in $\mathsf{CIRL}$ then we can describe subalgebras in a clear way:
$$
\SU (\bigoplus_{i \in I} \alg A_i) = \{\bigoplus_{i \in I} \alg B_i: \alg B_i \in \SU(\alg A)\}.
$$
However, if we consider sums in $\mathsf{FL_{ew}}$, the constant $0$ is represented by the $0$ of the lowermost component and so the above expression is equivocal.
Therefore, when we write $\bigoplus_{i \in I} \alg A_i$ we will always regard the algebras $\alg A_i$, $i >0$ as algebras in $\mathsf{CIRL}$, i.e. their zero-free reducts.

At present, we do not have a characterization of which subvarieties of $\mathsf{FL_{ew}}$ are closed under ordinal sums, nor which equations are preserved by ordinal sums, however the following can be easily shown.
\begin{lemma}\label{lemma:ordsumpres}
	The following are preserved by the ordinal sum construction:
\begin{enumerate}
\item all join-free equations in one variable,
\item the divisibility equation (div).
\end{enumerate}	
\end{lemma}
We have the following.
\begin{proposition}\label{prop:ordsumexists}
	Ordinal sums always exist in:
	\begin{enumerate}
	\item $\mathsf{FL_{ew}}$,
	\item $\mathsf{HL}$,
	\item the class of finite algebras in
$\mathsf{CIRL}$,
\item the class of totally ordered algebras in $\mathsf{CIRL}$.
\end{enumerate}
\end{proposition}
\begin{proof}
	The claim follows from Lemma \ref{lemma:ordsumpres}, and the fact that both $\mathsf{FL_{ew}}$-algebras and finite $\mathsf{CIRL}$s are always bounded.
\end{proof}

Hence, $\mathsf{FL_{ew}}$ is closed under ordinal sums and so is $\mathsf{HL}$. An algebra in $\mathsf{FL_{ew}}$ is called {\em $n$-potent} (for $n \in \mathbb N$) if it satisfies the equation $x^n \app x^{n-1}$; if $n=2$ we use the term {\em idempotent}.  The variety of $n$-potent $\mathsf{FL_{ew}}$-algebras is called $\mathsf{P_{n}FL_{ew}}$ (and in the literature it is sometimes shortened in $\mathsf{E_{n}}$), and we shall call $\mathsf{P_{n}HL}$ the largest $n$-potent subvarieties of  $\mathsf{HL}$; it follows from Proposition \ref{prop:ordsumexists} that they are both closed under ordinal sums.

However, in general a proper subvariety of $\mathsf{CIRL}$ or $\mathsf{FL_{ew}}$ is not closed under ordinal sums. For instance, we can easily construct an ordinal sum of two $\mathsf{MTL}$-algebras that is not in $\mathsf{MTL}$.

\begin{example}\label{ex:notsum}
	Consider the ordinal sum $\alg 4 \oplus \alg 2$ of the four-element Boolean algebra $\alg 4$ and the two-element Boolean algebra $\alg 2$ (in Figure \ref{figure:failprel}). Prelinearity fails since the join $(a \to b) \lor (b \to a)$ in the ordinal sum is redefined to be the lowest element of $\alg 2$, $0_{2}$, and thus it is not $1$. \begin{figure}[htbp]
\begin{center}
\begin{tikzpicture}
\draw (0,0) -- (-1,1) -- (0,2) -- (0,3);
\draw (0,0) -- (1,1) -- (0,2);
\draw[fill] (0,0) circle [radius=0.05];
\draw[fill] (-1,1) circle [radius=0.05];
\draw[fill] (0,2) circle [radius=0.05];
\draw[fill] (0,3) circle [radius=0.05];
\draw[fill] (1,1) circle [radius=0.05];
\node[right] at (0,0) {\footnotesize $0$};
\node[left] at (-1,1) {\footnotesize $a$};
\node[right] at (1,1) {\footnotesize $b$};
\node[right] at (0,2) {\footnotesize $0_{2}$};
\node[right] at (0,3) {\footnotesize $1$};
\end{tikzpicture}
\end{center}
\caption{$\alg 4 \oplus \alg 2$}\label{figure:failprel}
\end{figure}
\end{example}
The previous example actually shows the following.

\begin{proposition} No nontrivial subvariety of $\MTL$ is closed under ordinal sums.
\end{proposition}
\begin{proof}
The claim follows from Example \ref{ex:notsum}, and the fact that Boolean algebras are the atom in the lattice of subvarieties of $\MTL$.
\end{proof}

However it can be easily checked that any (possibly infinite) ordinal sum of totally ordered $\MTL$-algebras is a totally ordered $\MTL$-algebra, and any (possibly infinite) ordinal sum of totally ordered $\BL$-algebras is a totally ordered $\BL$-algebra (since it is an ordinal sum of totally ordered Wajsberg hoops \cite{AglianoMontagna2003}).
We conclude the preliminaries with the following observation.
\begin{lemma}\label{closed}  For a subvariety $\vv V$ of $\mathsf{FL_{ew}}$ the following are equivalent:
\begin{enumerate}
\item $\vv V$ is closed under finite ordinal sums;
\item $\vv V$ is closed under ordinal sums.
\end{enumerate}
\end{lemma}

\begin{proof} Clearly 2 implies 1.
Suppose now that $\vv V$ is not closed under ordinal sums. Thus, there is a family $(\alg A_i)_{i \in I}$ such that $\alg A_i \in \vv V$ for all $i \in I$ but
$\alg A = \bigoplus_{i \in I} \alg A_i \notin \vv V$. Then it must exists an equation $p(\vuc xn) \app 1$ holding in $\vv V$ and elements $\vuc an \in \bigoplus_{i \in I} \alg A_i$ such that $p(\vuc an) \ne 1$. Let $\alg A_{i_0},\dots,\alg A_{i_k}$ be an enumeration of all the components of the ordinal sum that contain at least one of the $\vuc an$. If $i_0 = 0$, then $\bigoplus_{j=0}^k \alg A_{i_j}$ is a finite ordinal sum belonging to $\vv V$ (since it is a subalgebra of $\alg A$) in which the equation fails; if $i_0\ne 1$ then $\alg A_0 \oplus \bigoplus_{j=0}^k \alg A_{i_j}$ is a finite ordinal sum with the same property. In any case $\vv V$ is not closed under finite ordinal sums and thus 1. implies 2.
\end{proof}

\section{Projectivity in $\mathsf{FL_{ew}}$}
Given Proposition \ref{prop:ordsumexists}, we will now make use of the ordinal sum construction to obtain some general results about projective algebras in $\mathsf{FL_{ew}}$.
\begin{lemma}\label{finiteprojective} Let $\vv K$ be a class of $\mathsf{FL_{ew}}$-algebras with the following properties:
\begin{enumerate}
\item There is a non-trivial algebra in $\vv K$;
\item $\vv K$ is closed under ordinal sums.
\end{enumerate}
If $\alg A \in \vv K$ is projective in $\vv K$ then $1$ is join irreducible in $\alg A$.  If $\alg A$ is also finite, then it is subdirectly irreducible.
\end{lemma}
\begin{proof}  
Let $\alg A$ be projective in $\vv K$ and $\alg B$ a non-trivial algebra in $\vv K$. Then $f\colon \alg A \oplus \alg B \longmapsto \alg A$ defined by $f(a) =a$ if $a \in A \setminus \{1\}$ and $f(d)=1$ if $d \in B$ is a surjective homomorphism. Since $\alg A$ is projective and $\alg A \oplus \alg B \in \vv K$, there is a homomorphism $g\colon \alg A \longmapsto \alg A \oplus \mathbf B$ such that $fg=id_A$. 

In order to prove that $1$ is join irreducible in $\alg A$, we show that if $a, b \in A \setminus \{1\}$, then $a \lor b \neq 1$. Let then $a, b \in A \setminus \{1\}$. Since $fg(a) = a, fg(b) = b$, by definition of $f$ we have that $g(a), g(b) \in A \setminus \{1\}$. Since $\alg B$ is non-trivial, there is $c \in B \setminus \{1\}$. Thus by definition of the order on the ordinal sum $\alg A \oplus \alg B$, $g(a) < c$ and $g(b) < c$, hence $g(a) \lor g(b) = g(a \lor b) \leq c < 1$. Since $g$ is a homomorphism, necessarily $a \lor b < 1$. Thus, if $1$ is a binary join of elements, at least one of the elements is $1$.

%Let $2=\{d,1\}$; then $f\colon \alg A \oplus \mathbf 2\longmapsto \alg A$ defined by $f(a) =a$ if $a \in A \setminus \{1\}$ and $f(d) =f(1) =1$ is a surjective homomorphism. Since $\alg A$ is projective and $\alg A \oplus \mathbf 2 \in \vv K$, there is a homomorphism $g\colon \alg A \longmapsto \alg A \oplus \mathbf 2$ such that $fg=id_A$. If $a \join b = 1$ in $\alg A$, then $g(a) \join g(b) =1$ and since $1$ join irreducible in $\alg A \oplus \mathbf 2$, either $g(a)=1$ or $g(b)=1$; then either $a =fg(a)=1$ or $b=fg(b) =1$ so $1$ is join irreducible in $\alg A$.

If $\alg A$ is finite then $1$ is completely join irreducible in $\alg A$; but this is well-known to being equivalent to subdirect irreducibility for finite members of $\mathsf{FL_{ew}}$ (see for instance \cite{GJKO}, Lemma 3.59).
\end{proof}

Since $\mathbf 2$ belongs to any non-trivial subvariety of $\mathsf{FL_{ew}}$ (indeed, it is the free algebra on the empty set of generators), given $\vv V$ a non-trivial subvariety of $\mathsf{FL_{ew}}$ that is closed under ordinal sums, the finite projective algebras in $\vv V$ must be subdirectly irreducible.

We will now highlight the connection between projective algebras and splitting algebras. Given a lattice $(L, \land, \lor)$, a {\em splitting pair} $(a,b)$ of elements of $L$ is such that $a \not\leq b$ and for any $c \in L$, either $a \leq c $ or $c \leq b$. This notion can be considered on lattices of subvarieties of a variety $\vv V$, where a splitting pair is then a pair of subvarieties of $\vv V$, $(\vv V_1, \vv V_2)$, where $\vv V_1$ is relatively axiomatized by a single equation and $\vv V_2$ is generated by a single finitely generated subdirectly irreducible algebra, called a {\em splitting algebra} (see \cite[Section 10]{GJKO} for more details).

\begin{lemma}\label{splitting} Let $\vv V$ be any variety and let $\alg A$ be an algebra that is subdirectly irreducible and projective in $\vv V$. Then
\begin{enumerate}
\item\label{splitting1} $\vv U = \{ \alg B \in \vv V: \alg A \notin \SU(\alg B) \}$ is a subvariety of $\vv V$;
\item\label{splitting2} for any subvariety $\vv W$ of $\vv V$ either $\vv W \sse \vv U$ or $\vv V(\alg A) \subseteq\vv W$.
\item\label{splitting3} there is an equation $\sigma$ in the language of $\vv V$ such that $\alg B \in \vv U$ if and only if $\alg B \vDash \sigma$;
\end{enumerate}
In other words, $(\vv U, \vv V(\alg A))$ constitutes a splitting pair in the lattice of subvarieties of $\vv V$.
\end{lemma}
\begin{proof} For \ref{splitting1}., $\vv U$ is closed under subalgebras by definition, under direct products since $\alg A$ is subdirectly irreducible, and under homomorphic images since $\alg A$ is projective.

For \ref{splitting2}., let $\vv W$ be a subvariety of $\vv V$. If $\alg A \in \vv W$, then  $\vv V(\alg A) \subseteq\vv W$. Otherwise, if $\alg A \notin \vv W$ we get that  $\vv W \subseteq \vv U$, since every algebra $\alg C$ in $\vv W$ is such that $\alg C \in \vv V$ and $\alg A \notin \SU(\alg C)$.

Notice that \ref{splitting2}. implies that the two varieties $\vv U$ and $\vv V(\alg A)$ constitute a splitting pair in the lattice of subvarieties of $\vv V$, which implies that $\vv U$ is axiomatized by a single identity, thus \ref{splitting3}. holds.
\end{proof}

Combining Lemma \ref{finiteprojective} and Lemma \ref{splitting}, we obtain the following.

\begin{theorem}\label{cor:projectivesplitting}
If $\vv V$ is a subvariety of $\mathsf{FL_{ew}}$ closed under ordinal sums, then any finite projective algebra in $\vv V$ is splitting in $\vv V$.
\end{theorem}

On the other hand, the converse holds only in very special cases. Splitting algebras have been thoroughly investigated  in many subvarieties of $\mathsf{FL_{ew}}$ (\cite{Agliano2017c}, \cite{Agliano2018a}, \cite{Agliano2018b}, \cite{AglianoUgolini2019a}), but the seminal paper on the subject is \cite{KowalskiOno2000}.

 We recall that a variety has the {\em finite model property} (or FMP) for its equational theory if and only if it is generated by its finite algebras. As a consequence of the general theory of splitting algebras (see again \cite[Section 10]{GJKO}), if a variety $\vv V$ has the FMP, every splitting algebra in $\vv V$ is a finite subdirectly irreducible algebra.
Thus, in subvarieties of $\mathsf{FL_{ew}}$ closed under ordinal sums and with the FMP, the study of finite projective algebras is particularly relevant for the study of splitting algebras. The finiteness hypothesis cannot be removed, as we will see in Section \ref{heytingsection} below.

The problem of finding splitting algebras in $\mathsf{FL_{ew}}$ is solved by Kowalski and Ono.
\begin{theorem} \cite{KowalskiOno2000} The only splitting algebra in $\mathsf{FL_{ew}}$ is $\mathbf 2$.
\end{theorem}
Since $\mathbf 2$ is the free algebra over the empty set of generators of any nontrivial subvariety of $\mathsf{FL_{ew}}$, we get the following fact.
\begin{lemma}\label{lemma:2proj}
$\alg 2$ is projective in every nontrivial subvariety of $\mathsf{FL_{ew}}$.
\end{lemma}
% Since  $\mathsf{FL_{ew}}$ has the finite model property \cite{OkadaTerui1999} we have:
Combining the two previous results with Theorem \ref{cor:projectivesplitting}, we get the following interesting fact.
\begin{theorem} The only finite projective algebra in $\mathsf{FL_{ew}}$ is $\mathbf 2$.
\end{theorem}

Interestingly, (the $0$-free reduct of) $\alg 2$ is instead not projective in $\mathsf{CIRL}$, since in \cite{Ugolini2022} it is shown that it is not projective in the variety of Wajsberg hoops, and as we previously noticed, projectivity is preserved in subvarieties.

By Proposition \ref{prop:ordsumexists}, $\mathsf{HL}$ is another variety closed under ordinal sums.
$\mathsf{HL}$ has the finite model property \cite{BlokFerr1993} and moreover:
\begin{theorem}\cite{Agliano2018a} An algebra is splitting in $\mathsf{HL}$ if and only if it is isomorphic with $\alg A \oplus \mathbf 2$ for some finite $\alg A \in \mathsf{HL}$.
\end{theorem}
Thus we can combine the previous theorem with Theorem \ref{cor:projectivesplitting}:
\begin{corollary}
All the finite projective algebras in $\mathsf{HL}$ are isomorphic with $\alg A \oplus \mathbf 2$ for some finite $\alg A \in \mathsf{HL}$.
\end{corollary}

We close this section with some other general considerations about projectivity and ordinal sums in $\mathsf{FL_{ew}}$ and in its $n$-potent subvarieties that will be useful in the rest of the paper.

\begin{lemma}\label{techlemma1} Let $\vv V$ be a subvariety of $\mathsf{FL_{ew}}$ with the finite model property and let $\alg A$ be projective in  $\vv V$; if $C$ is an infinite subset of $A$
closed under $\imp$, then the least upper bound of $C \setminus\{1\}$ exists and it is equal to $1$.
\end{lemma}
\begin{proof} Assume that this is not the case, thus either the least upper bound does not exist or it is not $1$. In any case $C \setminus \{1\}$ has (at least) an upper bound $a < 1$. Since $\alg A$ is projective it is a retract of a suitable
free algebra in $\vv V$, say $\alg F_\vv V(X)$; so there is a surjective homomorphism $f\colon \alg F_\vv V(X) \longrightarrow \alg A$ and a monomorphism $g\colon \alg A \longrightarrow \alg F_\vv V(X)$ such that $fg=id_\alg A$. Let $t(\vuc xn) = g(a)$; if $t(\vuc xn) =1$, then $1=f(t(\vuc xn)) = fg(a) =a$, a contradiction.  Thus $t(\vuc xn) \ne 1$ and so the equation $t(\vuc xn) \app 1$ must fail in $\vv V$; since $\vv V$ has the finite model property, there must be a finite $\alg B \in \vv V$ in which $t(\vuc xn) \app 1$ fails. This is equivalent to saying that there is a surjective homomorphism $h\colon\alg F_\vv V(X) \longrightarrow \alg B$ with $h(t(\vuc xn)) \ne 1$. Now $g(C)$ is infinite, since $g$ is a monomorphism, and $\alg B$ is finite; so there are $s, r\in g(C)$ such that $s \ne r$ and $h(s) = h(r)$. Without loss of generality, we assume that $r \not\le s$, then $r \imp s \ne 1$ so $r \imp s \in g(C \setminus\{1\})$ and thus $t(\vuc xn) \ge r \imp s$. In conclusion
$$
h(t(\vuc xn)) \ge h(r \imp s) = h(r) \imp h(s) = 1,
$$
thus $h(t(\vuc xn)) = 1$, that is a contradiction. Hence the thesis follows.
\end{proof}

As a consequence we get the following.

\begin{proposition}\label{sumprojective}  Let $\vv V$ be a subvariety of $\mathsf{FL_{ew}}$ with the finite model property and let $\alg A, \alg B \in \vv V$, with $\alg B$ nontrivial, be such that
$\alg A \oplus \alg B$ is projective in $\vv V$ (and so $\alg A \oplus \alg B \in \vv V$).  Then $\alg A$ is finite.
\end{proposition}
\begin{proof} $A$ is a subset of $\alg A \oplus \alg B$ closed under $\imp$ such that the least upper bound of $A \setminus \{1\}$ is not $1$; by Lemma \ref{techlemma1}, $\alg A$ must be finite.
\end{proof}

In general if $\alg A, \alg B, \alg A \oplus \alg B \in \vv V$ and $\alg A \oplus \alg B$ is projective, $\alg B$ is not necessarily projective; however there is a weakening of the notion that is very useful in subvarieties of $\mathsf{FL_{ew}}$. We call an algebra $\alg A \in \vv V$ {\em zero-projective} in $\vv V$ if for any $\alg B,\alg C \in \vv V$, any homomorphism $h\colon\alg A \longrightarrow \alg C$  and any surjective homomorphism $g\colon \alg B \longrightarrow \alg C$ such that $g^{-1}(0) = \{0\}$, there is a homomorphism
$f\colon \alg A \longrightarrow \alg B$ with $h=gf$.

\begin{lemma}\label{weaklyprojective} Let $\vv V$ be a subvariety of $\mathsf{FL_{ew}}$ closed under ordinal sums. If  $\alg A_i\in \vv V$ for all $i=0,\dots,n$ and $\bigoplus_{i=0}^n \alg A_i$ is projective, then $\alg A_i$ is zero-projective for all $i= 0,\dots, n$.
\end{lemma}
\begin{proof} Fix $i \le n$ and let $\alg U = \bigoplus_{j=0}^{i-1} \alg A_j$ and $\alg V = \bigoplus_{j=i+1}^n \alg A_j$; note that $\alg U$, $\alg V$ may be trivial if $i=0$ or $i=n$ but in any case $\bigoplus_{j=0}^n \alg A_j = \alg U \oplus \alg A_i \oplus \alg V$.  Suppose that there are $\alg B,\alg C \in \vv V$, a homomorphism $h\colon\alg A_i \longrightarrow \alg C$, and an onto homomorphism $g\colon\alg B \longrightarrow \alg C$ such that $g^{-1}(0) = \{0\}$. We will show that there is a homomorphism
$f\colon \alg A_{i} \longrightarrow \alg B$ with $h=gf$.
Consider the map $g'\colon \alg U \oplus \alg B \longrightarrow \alg U \oplus \alg C $ defined by $g'(x) = x$ if $x \in  U \setminus \{1\}$ and $g'(x)=g(x)$ otherwise; it follows from Proposition 3.2 in \cite{AglianoMontagna2003} that $g'$ is an onto homomorphism. Similarly we can define a homomorphism $h'\colon \alg U \oplus \alg A_i \oplus \alg V \longrightarrow \alg U \oplus \alg C$, as $h'(x) = x$ if $x \in  U \setminus \{1\}$, $h'(x) = h(x)$ if $x \in A_{i}$ and $h(x) = 1$ otherwise. Since $\alg U \oplus \alg A_i \oplus \alg V = \bigoplus_{j=0}^n \alg A_j $ is projective, there is an $f'\colon \alg U \oplus \alg A_i  \oplus \alg V \longrightarrow \alg U \oplus \alg B $ with $h'=g'f'$. Now $g'f'(0_{\alg A_i}) = h'(0_{\alg A_i})= h(0_{\alg A_i}) = 0_\alg C$. Given that $g'^{-1}(0_{\alg C}) = g^{-1}(0_{\alg C}) = \{0_{\alg B}\}$, we get that  $f'(0_{\alg A_{i}}) = 0 _\alg B$;  thus the restriction of $f'$ to $\alg A_{i}$, $f\colon \alg A_i \longrightarrow \alg B$  is a homomorphism.  Thus $gf=h$, and $\alg A_i$ is zero-projective.
\end{proof}

Moreover, we have the following.
\begin{lemma}\label{weaklyprojective2} Let $\vv V$ be a subvariety of  $\mathsf{CIRL}$; if $\alg A \in \vv V$ and $\mathbf 2 \oplus \alg A$ is zero-projective in the class $\vv K =\{\mathbf 2 \oplus \alg B: \alg B \in \vv V\}$, then $\alg A$ is projective in $\vv V$.
\end{lemma}
\begin{proof} Let $\alg B,\alg C \in \vv V$, $h\colon \alg A \longrightarrow \alg C$ and $g\colon \alg B \longrightarrow \alg C$ a surjective homomorphism. Define $h'\colon \mathbf 2 \oplus \alg A \longrightarrow \mathbf 2 \oplus \alg C$ and $g'\colon \mathbf 2 \oplus \alg B \longrightarrow \mathbf 2 \oplus \alg C$ as in the proof of Lemma \ref{weaklyprojective}, i.e. to coincide with, respectively, $h$ on $\alg A$ and $g$ on $\alg B$ and be the identity on $\alg 2$; by definition $g'^{-1}(0) = \{0\}$ and so, since $\mathbf 2 \oplus \alg A$ is zero-projective in $\vv K$, there is a homomorphism $f'\colon \mathbf 2 \oplus \alg A \longrightarrow \mathbf 2 \oplus \alg B$ with $g'f'=h'$.  Notice that $f'(0_{\alg A}) \ge 0_{\alg B}$, indeed if otherwise $f'(0_{\alg A}) = 0$, then $g'f'(0_{\alg A}) = g'(0) = 0 \ne h'(0_\alg A)$; thus $f'(0_{\alg A}) \ge 0_{\alg B}$ which implies that the restriction of $f'$ to $\alg A$, let us call it $f$, is a homomorphism form $\alg A$ to $\alg B$. Since $gf=h$,  $\alg A$ is projective in $\vv V$.
\end{proof}

\section{Heyting algebras}\label{heytingsection}
In this section we apply our general result on ordinal sums to study projectivity in the variety of {\em Heyting algebras}. Heyting algebras are the idempotent algebras in $\mathsf{FL_{ew}}$, and their variety $\mathsf{HA}$ is the equivalent algebraic semantics of {\em Brouwer's Intuitionistic Logic}. Heyting algebras are divisible by Lemma \ref{Heyting}(5) below, and can alternatively be characterized as the class of  $\mathsf{HL}$-algebras satisfying the further equation
$$
xy \app x \meet y,
$$
i.e. $\mathsf{HA} = \mathsf{P_{2}FL_{ew}} = \mathsf{P_{2}HL}$.

%In what follows we also deal in particular with zero-free subreducts of Heyting algebras. Such algebras are commonly called {\em Brouwerian algebras}, and we denote their variety by $\mathsf{Br}$; bounded Brouwerian algebras are polynomially equivalent to Heyting algebras.
The fact that the product and the meet coincide forces many equations to hold in $\mathsf{HA}$; we collect some of them (without proofs) in the following.

\begin{lemma}\label{Heyting}  Let $\alg A$ be a Heyting algebra and $a,b,c \in A$; then
\begin{enumerate}
\item if $a \le c$ and $a \imp b =b$, then $c \imp b=b$;
\item $a \le (a \imp b) \imp b)$;
\item $((a \imp b) \imp b) \imp (a \imp b) = a \imp b$;
\item $(a \imp b) \imp ((a \imp b) \imp b) = ((a \imp b) \imp b)$;
\item if $a \le b$, then $a \meet (b \imp c) = a \meet c$.
\end{enumerate}
\end{lemma}

%Note that Lemma \ref{Heyting}(5) implies that Heyting algebras are pseudocomplemented; since the variety of Heyting algebras is closed under ordinal sums $\mathbf 4 \oplus \mathbf 2$ is a non Stonean Heyting algebra.

 Now, every Heyting algebra $\alg A$ is the ordinal sum of sum irreducible Heyting algebras. If $\alg A$ is also finite and projective, then the last component must be $\mathbf 2$,  since by Lemma \ref{finiteprojective} $\alg A$ is also subdirectly irreducible, and subdirectly irreducible Heyting algebra are exactly the algebras with $\alg 2$ as the last component (folklore, see \cite{BuSa}).

 We first show that characterizing the finite sum irreducible Heyting algebras is deceptively easy. We call an element $a$ of a poset $\la P, \le\ra$ a {\em node} if it is a conical element: for all $b \in P$ either $b \le a$ or $a \le b$; for any $a \in P$ the {\em upset} of $a$ is the set $\mathop{\uparrow} a = \{b: a \le b\}$. The proof
of the following lemma is straightforward and we leave it as an exercise.

\begin{lemma}\label{node} Let $\alg A$ be an Heyting algebra and let $a$ be a node. Then:
\begin{enumerate}
\item $\alg A_a = \la (A \setminus \mathop{\uparrow} a) \cup \{1\}, \join,\meet,\imp,0,1\ra$ is a Heyting algebra where
$$
u \join v = \left\{
              \begin{array}{ll}
                1, & \hbox{if $u \join v = a$;} \\
                u \join_\alg A v, & \hbox{ otherwise.}
              \end{array}
            \right.;
$$
\item $\alg A^a = \la \mathop{\uparrow} a, \join,\meet,\cdot,\imp,1\ra$ is a bounded Brouwerian algebra;
\item $\alg A = \alg A_a \oplus \alg A^a$.
\end{enumerate}
\end{lemma}

\begin{proposition}
	A finite Heyting algebra is sum irreducible if and only if it has no node different from $0,1$.
\end{proposition}
\begin{proof}  If there is a node different from $0,1$ then the algebra is sum reducible by Lemma \ref{node}. Vice versa if $\alg A$ is finite and sum reducible, i.e. $\alg A = \alg B \oplus \alg C$ nontrivially, then the minimum $c \in C$ is a node of $\alg A$ different from $0,1$.
\end{proof}

Since in a totally ordered Heyting algebra every element is a node, from Lemma \ref{node} we also obtain the following.

\begin{proposition}
 A totally ordered Heyting algebra is an ordinal sum of copies of $\mathbf 2$.
\end{proposition}

Every finite Boolean algebra is a sum irreducible Heyting algebra;  but since every finite distributive lattice can be given the structure of an Heyting algebra, there are many non-Boolean sum irreducible Heyting algebras. However the list becomes very short if we consider  sum irreducible Heyting algebras that are irreducible components of a finite projective Heyting algebra. We are now going to show that, indeed, the only possible components are $\alg 2$ and $\alg 4$. In order to use Lemma \ref{weaklyprojective}, we start by studying finite zero-projective Heyting algebras.

In the following theorem we will make use of the free Heyting algebra on one generator $\alg F_\mathsf{HA}(x)$, also called the {\em Nishimura lattice}, that is an infinite Heyting algebra totally described in \cite{Nishimura1960}. In Figure \ref{Nishimura} we see a picture of the bottom of the lattice. Moreover, we will denote by $\mathbf 3$ and $\mathbf 4$ the three and four elements Heyting algebras respectively, where $\mathbf 3 \cong \mathbf 2 \oplus \mathbf 2$.

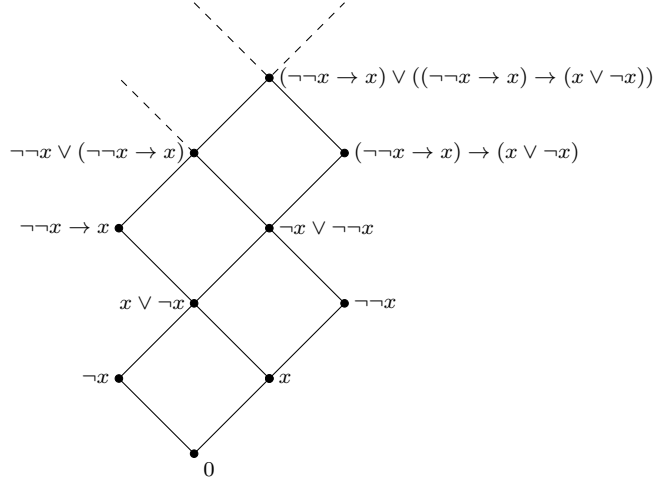
\begin{figure}[htbp]
\begin{center}
\begin{tikzpicture}
\draw (1,3)-- (2,4) -- (1,5) -- (0,4) -- (-1,3) -- (0,2) -- (1,3) -- (0,4) -- (-1,5) -- (0,6) -- (1,5) --(2,6) -- (1,7) -- (0,6);
\draw[dashed] (0,6) -- (-1,7);
\draw[dashed] (0,8) -- (1,7) -- (2,8);
\draw[fill] (1,3) circle [radius=0.05];
\draw[fill] (2,4) circle [radius=0.05];
\draw[fill] (1,5) circle [radius=0.05];
\draw[fill] (0,4) circle [radius=0.05];
\draw[fill] (-1,3) circle [radius=0.05];
\draw[fill] (0,2) circle [radius=0.05];
\draw[fill] (1,3) circle [radius=0.05];
\draw[fill] (-1,5) circle [radius=0.05];
\draw[fill] (0,6) circle [radius=0.05];
\draw[fill] (2,6) circle [radius=0.05];
\draw[fill] (1,7) circle [radius=0.05];
\node[right] at (0,1.8) {\footnotesize $0$};
\node[right] at (1,3) {\footnotesize $x$};
\node[left] at (-1,3) {\footnotesize $\neg x$};
\node[left] at (0,4) {\footnotesize $x \join \neg x$};
\node[right] at (2,4) {\footnotesize $\neg\neg x$};
\node[left]  at (-1,5) {\footnotesize $\neg\neg x \imp x$};
\node[right] at (1,5) {\footnotesize $\neg x \join \neg\neg x$};
\node[left]  at (0,6) {\footnotesize $\neg\neg x \join (\neg\neg x \imp x)$};
\node[right] at (2,6) {\footnotesize $(\neg\neg x \imp x) \imp (x \join \neg x)$};
\node[right] at (1,7) {\footnotesize $(\neg\neg x \imp x) \join ((\neg\neg x \imp x) \imp (x \join \neg x))$};
\end{tikzpicture}
\end{center}
\caption{The bottom of the Nishimura lattice\label{Nishimura}}
\end{figure}

\begin{theorem}\label{twoatoms} Let $\alg A$ be a finite zero-projective Heyting algebra; if $\alg A$ is sum irreducible and has exactly two atoms
then $\alg A \cong \mathbf 4$.
\end{theorem}
\begin{proof} If $a,b$ are the only atoms of $\alg A$ it is enough to show that $a \join b =1$. Indeed, there cannot be elements below or incomparable to $a, b$ since they are the only atoms, and there cannot be elements above either of them or otherwise the lattice would not be distributive (there would be a sublattice isomorphic to $\alg N_{5}$). Suppose then  that $a \join b < 1$ in $\alg A$; since $\alg A$ is sum irreducible it cannot contain any node different from $0$ or $1$, so there has to be an element $d \in A$ covering $a$ incomparable with $ a \lor b$ (or the symmetric case). Let $F$ be the filter generated by $d$, $F = \mathop{\uparrow} d$,
 then it can be checked (see Lemma 4.7 in \cite{BalbesHorn1970}) that, if $\th$ is the congruence associated with $F$, then $\alg A /\th \cong \mathbf 3$.
In particular, if we denote by $\{0,u,1\}$ the universe of $\mathbf 3$, there is a onto homomorphism $h_0\colon \alg A \longrightarrow\mathbf 3$ with $h_0(a) = u$. Similarly, since $b$ is an atom (and hence the principal filter generated by $b$ is maximal) there is a onto homomorphism $h_1 \colon\alg A \longrightarrow \mathbf 2$ with $h_1(a) =0$. It follows that if we define $h(y) = (h_0(y),h_1(y))$ then $h\colon \alg A \longrightarrow \mathbf 3 \times \mathbf 2$ is a onto homomorphism and moreover $(u,0)$ generates $\mathbf 3 \times \mathbf 2$ (see Figure \ref{3times2}).

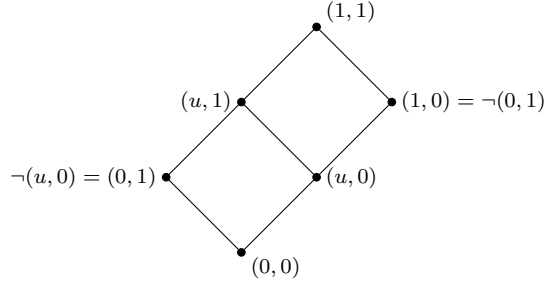
\begin{figure}[htbp]
\begin{center}
\begin{tikzpicture}
\draw (1,3)-- (2,4) -- (1,5) -- (0,4) -- (-1,3) -- (0,2) -- (1,3) -- (0,4);
\draw[fill] (1,3) circle [radius=0.05];
\draw[fill] (2,4) circle [radius=0.05];
\draw[fill] (1,5) circle [radius=0.05];
\draw[fill] (0,4) circle [radius=0.05];
\draw[fill] (-1,3) circle [radius=0.05];
\draw[fill] (0,2) circle [radius=0.05];
\node[right] at (0,1.8) {\footnotesize $(0,0)$};
\node[right] at (1,3) {\footnotesize $(u,0)$};
\node[left] at (-1,3) {\footnotesize $\neg(u,0) = (0,1)$};
\node[left] at (0,4) {\footnotesize $(u,1)$};
\node[right] at (2,4) {\footnotesize $(1,0)=\neg(0,1)$};
\node[right] at (1,5.2) {\footnotesize $(1,1)$};
\end{tikzpicture}
\end{center}
\caption{The algebra $\mathbf 3 \times \mathbf 2$\label{3times2}}
\end{figure}

Let now $f$ be the onto homomorphism from the Nishimura lattice $\alg F_\mathsf{HA}(x)$ to $\mathbf 3 \times \mathbf 2$ defined by $f(x) = (u,0)$. Observe that
\begin{align*}
&f(\neg x) = \neg(u,0) = (0,1)\\
&f(\neg\neg x) = \neg\neg(u,0) = (1,0)\\
& f(x \join \neg x) = (u,0) \join \neg(u,0) = (u,0) \join (0,1) = (u,1).
\end{align*}
Thus, by looking at Figure \ref{Nishimura}, by order preservation $f^{-1}(0,0) = \{0\}$ and $f^{-1}(u,0) = \{x\}$.
Since $\alg A$ is zero-projective, there is a homomorphism $g\colon \alg A \longrightarrow \alg F_\mathsf{HA}(x)$ with $fg = h$. Now, $fg( a) = h( a) = (h_0(a),h_1( a)) = (u,0)$; since $f^{-1}(u,0) = \{x\}$, it means that $g(a) = x$. Since $x$ generates $\alg F_\mathsf{HA}(x)$, the map $g$ is onto; but the Nishimura lattice is infinite and  $\alg A$ is finite, a contradiction. Hence the thesis follows.
\end{proof}

It follows by Lemma \ref{Heyting}(5) that Heyting algebras are {\em pseudocomplemented}, that is, they satisfy the identity $x \land \neg x \approx 0$. Therefore, we can use the following fact. Given an algebra $\alg A$, let $R(\alg A)$ be the set of its {\em regular} elements, that is, elements $x \in A$ such that $\neg\neg x = x$.
\begin{theorem}\cite{Castanoetal2011}\label{finitefree} Let $\vv V$ be a subvariety of $\mathsf{FL_{ew}}$ and let $\alg F_\vv V(X)$ be the free algebra in $\vv V$ on $X$. If $\vv V$ is pseudocomplemented, i.e., each algebra in $\vv V$ is pseudocomplemented, then $R(\alg F_\vv V(X)) \cong \alg F_\vv B(\neg\neg X)$, the free Boolean algebra on the set $\neg\neg X = \{\neg\neg x: x \in X\}$. %If $\vv V$ is Stonean, then $\alg F_\vv B(X)$ is a retract of $\alg F_\vv V(X)$.
\end{theorem}
\begin{theorem}\label{nothree} Let $\alg A$ be a finite zero-projective Heyting algebra; then $\alg A$ has at most two atoms.
\end{theorem}
\begin{proof}  Let $\alg F = \alg F_\mathsf{HA}(x,y)$ be the free Heyting algebra on two generators; by Theorem \ref{finitefree} $R(\alg F)$ is the free Boolean algebra generated by the atoms
$$
b_1 =\neg\neg x \meet \neg\neg y\quad b_2= \neg\neg x \meet \neg y \quad b_3 = \neg x \meet \neg\neg y\quad b_4= \neg x \meet \neg y.
$$
It is a nice exercise to show that the subalgebra $\alg G$ of $\alg F$ generated by $b_1,b_2$ is infinite (see \cite[Theorem 4.4]{BalbesHorn1970}). We will show that if $\alg A$ has at least three atoms, then $\alg G$ is a homomorphic image of $\alg A$; since $\alg A$ is finite this is clear contradiction.

Let $a_1, \ldots, a_{n}$ be the atoms of $\alg A$, with $n \geq 3$. Set $h(a_{1}) = b_{1},  h(a_{2}) = b_{2}, h(a_{3}) = b_{3}$ and $h(a_{i}) = b_{3}$ for $i: 4\leq i \leq n$; then the map $h\colon \alg A \longrightarrow R(\alg F)$ defined by $h(x) = \bigvee\{h(a_{i}): a_i \le x\}$ is a homomorphism.
Now $\neg\neg\colon \alg F \longrightarrow R(\alg F)$ is an onto homomorphism and  moreover if $u \in F$, $u \le (u \imp 0) \imp 0 = \neg \neg u$ (Lemma \ref{Heyting}(2)); thus $\neg\neg u = 0$ if and only if $u=0$. Since $\alg A$ is zero-projective, there is a homomorphism $g\colon\alg A \longrightarrow \alg F$ with $\neg\neg g = h$.  But then $g(\neg\neg a_1) = \neg\neg g(a_1) =h(a_1) = b_1$ and
$g(\neg\neg a_2) = \neg\neg g(a_2) = h(a_2) = b_2$, so $g$ restricted to $\alg A$ is  onto $\alg G$. Thus we reached the desired contradiction.
\end{proof}

\begin{corollary}\label{twoandfour} Let $\alg A$ be a finite projective Heyting algebra and suppose that $\bigoplus_{i=0}^n \alg A_i$ is a decomposition of $\alg A$ into its sum irreducible components. Then $\alg A_n \cong\mathbf 2$ and for each $i <n$, $\alg A_i$ is isomorphic with either $\mathbf 2$ or $\mathbf 4$.
\end{corollary}
\begin{proof} We have already observed that $\alg A_n$ must be isomorphic with $\mathbf 2$. If $i<n$, by Lemma \ref{weaklyprojective} $\alg A_i$ is zero-projective, hence it can have at most two atoms by Theorem \ref{nothree}. If it has only one atom, then $\alg A_i \cong \mathbf 2$; if it has two atoms, since it also sum irreducible, Theorem \ref{twoatoms} applies and thus $\alg A_i \cong \mathbf 4$.
\end{proof}

We will now show that the necessary condition in Corollary \ref{twoandfour} is also sufficient. First we need a key lemma.

\begin{lemma} \label{inductionstep} Let $\alg A$ be a finite Heyting algebra, and let $\alg B$ be either $\alg 2$ or $\alg 4$. If $\alg A \oplus \mathbf 2$ is projective, then also $\alg A \oplus \alg B \oplus \mathbf 2$ is projective.\end{lemma}
\begin{proof} We shall label the elements as in Figure \ref{figure:lemmainduction}. If $\alg A \oplus \mathbf 2$ is (finite and) projective, then it is a retract of a free algebra generated by a (finite) set $X$, $\alg F_{\sf HA}(X)$. Thus, there exist homomorphisms $i\colon \alg A \oplus \alg 2 \to \alg F_{\sf HA}(X)$, and $j\colon \alg F_{\sf HA}(X) \to \alg A \oplus \alg 2$ such that $j \circ i = id_{\alg A \oplus \alg 2}$. We consider first the case $\alg B \cong 2$. Given a new variable $y$, we will define homomorphisms $\bar i\colon \alg A \oplus \alg 2 \oplus \alg 2 \to \alg F_{\sf HA}(X \cup \{y\})$, and $\bar j\colon \alg F_{\sf HA}(X \cup \{y\}) \to \alg A \oplus \alg 2 \oplus \alg 2$ such that $\bar j \circ \bar i = id_{\alg A \oplus \alg 2 \oplus \alg 2}$. We define:
\begin{align*}
&\bar{i}(x) = i(x), \quad \mbox{ for } x \in A \oplus 2\\
&\bar{i}(c_1) = y \join (y \imp i(c_0)).
\end{align*}

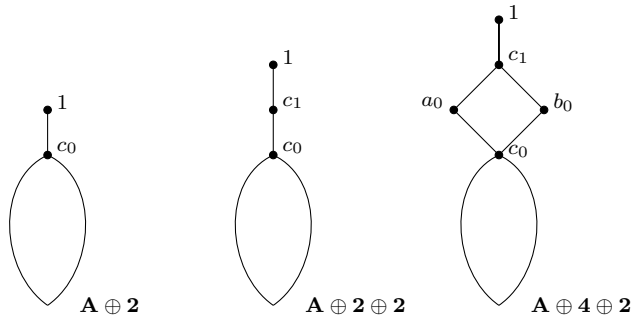
\begin{figure}[h]
\begin{center}
\begin{tikzpicture}
\draw (0,0) to [out=145, in= 205] (0,2);
\draw (0,0) to [out=35, in= 335] (0,2);
\draw (0,2) to (0,2.6);
\draw[fill] (0,2) circle [radius=0.05];
\draw[fill] (0,2.6) circle [radius=0.05];
\node[right] at (0,2.1) {\footnotesize $c_0$};
\node[right] at (0,2.7) {\footnotesize $1$};
%************************
\draw (3,0) to [out=145, in= 205] (3,2);
\draw (3,0) to [out=35, in= 335] (3,2);
\draw (3,2) to (3,3.2);
\draw[fill] (3,2) circle [radius=0.05];
\draw[fill] (3,2.6) circle [radius=0.05];
\draw[fill] (3,3.2) circle [radius=0.05];
\node[right] at (3,2.1) {\footnotesize $c_0$};
\node[right] at (3,2.7) {\footnotesize $c_1$};
\node[right] at (3,3.3) {\footnotesize $1$};
%*************************************
\draw (6,0) to [out=145, in= 205] (6,2);
\draw (6,0) to [out=35, in= 335] (6,2);
\draw[fill] (6,2) circle [radius=0.05];
\draw[fill] (5.4,2.6) circle [radius=0.05];
\draw[fill] (6.6,2.6) circle [radius=0.05];
\draw[fill] (6,3.2) circle [radius=0.05];
\draw[fill] (6,3.8) circle [radius=0.05];
\draw (6,2) -- (5.4,2.6) -- (6,3.2) -- (6,3.8) -- (6,3.2) -- (6.6,2.6)-- (6,2);
\node[right] at (6,2.05) {\footnotesize $c_0$};
\node[left] at (5.4,2.7) {\footnotesize $a_0$};
\node[right] at (6.6,2.7) {\footnotesize $b_0$};
\node[right] at (6,3.3) {\footnotesize $c_1$};
\node[right] at (6,3.9) {\footnotesize $1$};
%**************************
\node[right] at (0.3,0) {\footnotesize $\alg A \oplus \mathbf 2$};
\node[right] at (3.3,0) {\footnotesize $\alg A \oplus \mathbf 2 \oplus \mathbf 2$};
\node[right] at (6.3,0) {\footnotesize $\alg A \oplus \mathbf 4 \oplus \mathbf 2$};
\end{tikzpicture}
\end{center}
\caption{Labeling $\alg B$ and $\mathbf 2$}\label{figure:lemmainduction}
\end{figure}

Since $i(c_0) \le \bar{i}(c_1)$ it is easy to check that $\bar{i}$ is a lattice homomorphism; for instance
$$
\bar{i}(c_0) \join \bar{i}(c_1) = i(c_0) \join \bar{i}(c_1) = \bar{i}(c_1) = \bar{i}(c_0 \join c_1).
$$
For the implication there is only one nontrivial case  and we proceed to prove it:
\begin{align*}
\bar{i}(c_1) \imp \bar{i}(c_0) &=  (y \join (y \imp i(c_0))) \imp i(c_0) \\
&= (y \imp i(c_0)) \meet ((y \imp i(c_0)) \imp i(c_0)) \quad\text{by \cite[Theorem 3.10(2)]{GJKO}}\\
&= (y \imp i(c_0)) \meet i(c_0) \quad\text{by Lemma \ref{Heyting}(5)}\\
&= i(c_0) \quad\text{by the properties of residuated lattices}\\
&=\bar{i}(c_0) = \bar{i}(c_1 \imp c_0) \quad\text{by definition of ordinal sum}.
\end{align*}
We now let $\bar j$ be the homomorphism coinciding with $j$ on $\alg F_{\sf HA}(X)$, and such that $\bar j(y) = c_1$. We prove that $\bar j \circ \bar i = id_{\alg A \oplus \alg 2 \oplus \alg 2}$. If $x \in \alg A \oplus \alg 2$, then $\bar j (\bar i(x)) = \bar j(i(x)) = j(i(x)) = x$. Moreover $\bar j (\bar i(c_1)) = \bar j(y \lor (y \to i(c_0))) = \bar j(y) \lor (\bar j(y) \to \bar j(i(c_0))) = c_1 \lor (c_1 \to c_0) = c_1 \lor c_0 = c_1$. Thus $\alg A \oplus \alg 2 \oplus \alg 2$ is a retract of $\alg F_{\sf HA}(X \cup \{y\})$ and is therefore projective.

Let us now consider the case where $\alg B \cong \alg 4$. We will define $\widehat i\colon \alg A \oplus \mathbf 4 \oplus \mathbf 2 \to \alg F_{\sf HA}(X \cup \{y\})$ and $\widehat j\colon \alg F_{\sf HA}(X \cup \{y\}) \to \alg A \oplus \mathbf 4 \oplus \mathbf 2$ such that their composition is the identity on $\alg A \oplus \mathbf 4 \oplus \mathbf 2$. First we define:
\begin{align*}
&\widehat{i}(x) = i(x) \quad \mbox{ for } x \in A \oplus 2\\
&\widehat{i}(a_0) = (y \imp i(c_0)) \imp i(c_0)\\
&\widehat{i}(b_0) = y \imp i(c_0)\\
&\widehat{i}(c_1) = \widehat{i}(a_0) \join \widehat{i}(b_0).
\end{align*}
It can be shown that $\widehat{i}$ is a lattice homomorphism, indeed, by definition $\widehat{i}(c_1) = \widehat{i}(a_0) \join\, \widehat{i}(b_0)$, and moreover $\widehat{i}(a_0) \land\, \widehat{i}(b_0) = ((y \imp i(c_0)) \imp i(c_0)) \land (y \imp i(c_0)) = i(c_0) = \widehat{i}(a_0 \land b_0)$.

Next, observe that for $v \in A\setminus\{1\}$ and $x \in \{c_0,a_0,b_0,c_1\}$, $i(v) \leq i(c_0) \leq \widehat i(x)$, and $i(c_0) \to i(v) = i(c_0 \to v) = i(v)$. Thus by Lemma \ref{Heyting}(1) we get:
$$
\widehat{i}(x) \imp \widehat{i}(v) = \widehat{i}(x) \imp i(v) = i(v) = \widehat{i}(v)= \widehat{i}(x \imp v).
$$
Moreover:
\begin{align*}
\widehat{i}(a_0) \imp \widehat{i}(c_0) &= ((y \imp i(c_0)) \imp i(c_0)) \imp i(c_0) \\
&= y \imp i(c_0)  \quad \text{by the properties of residuated lattices}\\
&= \widehat{i}(b_0)= \widehat{i}(a_0 \imp c_0) \quad\text{by residuation}.
\end{align*}
Similarly, $\widehat{i}(b_0) \imp \widehat{i}(c_0)= \widehat{i}(b_0 \imp c_0)$. The arguments for proving $\widehat{i}(a_0) \imp \widehat{i}(b_0)= \widehat{i}(a_0 \imp b_0)$ and $\widehat{i}(b_0) \imp \widehat{i}(a_0)= \widehat{i}(b_0 \imp a_0)$ are shown in an analogous way, using Lemma \ref{Heyting}(3) and (4).
Finally, the proof that $\widehat{i}(c_1) \imp \widehat{i}(x)= \widehat{f}(c_1 \imp x)$ whenever $x < c_1$ follows from direct computation and it is left to the reader. Thus, $\widehat i$ is an homomorphism. We let $\widehat j$ be the homomorphism extending $j$ on $\alg F_{\sf HA}(X)$ and such that $\widehat j(y) = a_0$.
We prove that $\widehat j \circ \widehat i$ is the identity on $\alg A \oplus \alg 4 \oplus \alg 2$.
For $x \in \alg A \oplus \alg 2$, $\widehat j \circ \widehat i(x) = j \circ i (x) = x$.
\begin{align*}
\widehat j \circ \widehat i(a_0)&= \widehat j((y \to i(c_0)) \to i(c_0))\\ &= (\widehat j(y) \to j \circ i(c_0)) \to j \circ i(c_0) = (a_0 \to c_0) \to c_0 = a_0.\\
\widehat j \circ \widehat i(b_0)&= \widehat j(y \to i(c_0)) = \widehat j(y) \to j \circ i(c_0) = a_0 \to c_0 = b_0\\
\widehat j \circ \widehat i(c_1)&= \widehat j(\widehat i(a_0) \lor \widehat i(b_0)) = a_0 \lor b_0 = c_1.
\end{align*}
Thus, $\alg A \oplus \alg 4 \oplus \alg 2$ is a retract ot $\alg F_{\mathsf HA}(X \cup \{y\})$ and therefore is projective.
\end{proof}

Thus we get the main result of this section.

\begin{theorem}\label{projectiveHA} Let $\kappa \leq \omega$ be any ordinal and let $\alg A = \bigoplus_{n < \kappa} \alg A_n$ where $\alg A_n$ is either $\mathbf 2$ or $\mathbf 4$ for all $n \in  \mathbb N$. Then $\alg A \oplus \mathbf 2$ is projective in $\mathsf{HA}$.
\end{theorem}
\begin{proof}
Since $\alg 2$ is projective in $\mathsf{HA}$, the claim follows by induction on $\kappa$ using Lemma \ref{inductionstep}.
\end{proof}

Finite projective Heyting algebras have been characterized in \cite{BalbesHorn1970}; in this section we have given an alternative proof using the concepts of ordinal sum and sum irreducibility to get (we hope) a more streamlined and clear sequence of arguments.

\section{Divisible residuated lattices}\label{section:hoops}
In the previous sections we have exploited the ordinal sum construction, while in this section we will make use of a strong order property that characterizes a large and interesting class of residuated lattices: divisibility.
We observe that the results in this section are more conveniently stated for commutative and integral residuated {\em semilattices}, which we recall are the subreducts of commutative and integral residuated lattices and share a good deal of their theory with them (see \cite{Agliano2018a} for an extensive treatment, even for the noncommutative case).
First we make a general observation.

\begin{proposition}\label{prop:retractproj1}
Let $\vv K$ be a class of algebras. If every algebra $\alg B \in \vv K$ is a retract of every algebra $\alg A \in \vv K$ of which it is an homomorphic image, then every algebra in $\vv K$ is projective in $\vv K$.
\end{proposition}
\begin{proof}
Let $\vv K$ be a class of algebras satisfying the hypothesis of the Proposition. Consider any algebra $\alg C \in \vv K$, and suppose there is an homomorphism $h\colon \alg C \to \alg B$, and a surjective homomorphism $g\colon \alg A \to \alg B$. Since $\alg B$ is a homomorphic image of $\alg A$, it is also its retract. Thus there exists an injective homomorphism $f'\colon\alg B \to \alg A$ such that $gf' = id_{\alg B}$. Then we can consider the homomorphism $f\colon \alg C \to \alg A$, $f = f'h$. We get $gf = gf'h = h$, and thus $\alg C$ is projective in $\vv K$.
\end{proof}

A usual application of this fact is taking $\vv K$ as the class $\vv V_{fin}$ of finite algebras of a variety $\vv V$; thus if every finite $\alg A \in \vv V$ is a retract of any finite $\alg B \in \vv V$ of which it is a homomorphic image, then every finite member of $\vv V$ is projective in $\vv V_{fin}$. In this case it is customary to say that every finite member of $\vv V$ is {\em finitely projective}.  In case $\vv V$ is also locally finite, the finite algebras are exactly the finitely presented algebras, and thus every finitely presented algebra in $\vv V$ is projective in $\vv V$.

The class of $\{\imp,\cdot\}$-subreducts of commutative and integral residuated semilattices is the quasivariety $\mathsf{PO}$ of commutative and integral residuated partially ordered monoids, commonly known as {\em pocrims} \cite{BlokRaf1997}; the fact that they are indeed all subreducts is a consequence of a very general embedding theorem stated in \cite{OnoKomori1985}. Since the divisibility equation makes the $\meet$ a definable operation by
$
a \,\meet\, b = (a \imp b)a, 
$
 every divisible variety of residuated semilattices is a variety of pocrims. In particular, the variety of hoops is a variety of pocrims, and for hoops we have the following result.

\begin{lemma}\label{idempo} Let $\alg A$ be a hoop and let $u \in A$ be idempotent; then for all $a,b \in A$
\begin{align}
&u \imp (a \imp b) = (u \imp a) \imp (u \imp b)\\
&u \imp ab = (u \imp a)(u \imp b).
\end{align}
\end{lemma}

The two properties above, albeit similar, are quite different in nature; the first one can be proved through a standard computation \cite{EDPC3} while the second involves very complex calculations: it was first proved
in \cite{SpinksVeroff2004} via a computer-assisted proof. We can now prove the following.

\begin{theorem}\label{mth} 
Every finite hoop is finitely projective in the variety of hoops,  hence in any locally finite variety of hoops every finitely presented (equivalently, finite) algebra is projective.
\end{theorem}
\begin{proof} Let $\alg A, \alg B$ be finite hoops and let $g\colon\alg B \longrightarrow \alg A$ be a surjective homomorphism. Let $\th = \op{ker}(g)$ and $F=1/\th$; then $F$ is a filter of $\alg B$ and, since $\alg B$ is finite,
it has a minimum $u$ which is idempotent. Let $f(x)\colon u \imp x$; then by Lemma \ref{idempo} $f$ is an endomorphism of $\alg B$. It is also idempotent since for $b \in B$
$$
f(f(b)) = u \imp (u \imp b) = u^2 \imp b = u \imp b = f(b).
$$
Now if $b \in F$, then  $f(b) = u \imp b =1$; conversely if  $f(b) =1$ then  $u \imp b =1$, i.e. $u \le b$  and so  $b \in F$. In conclusion $F = 1/\op{ker}(f)$ and thus $\op{ker}(f) = \th$; if we set $\alg C = f(\alg B)$, which is a subalgebra of $\alg B$, then
$$
\alg A \cong \alg B /\th = \alg B/\op{ker}(f) \cong \alg C.
$$
Let now $h\colon \alg A \longmapsto \alg C$ be the resulting isomorphism, then $h$ is an injective homomorphism from $\alg A$ to $\alg B$. Moreover, by definition $h(a) = f(b)$ where $b$ is such that
$g(b) =a$; observe also that the fact that $f$ is idempotent implies that $(b,f(b)) \in \op{ker}(f)=\th$, thus $g(b) = g(f(b))$.
Hence, if $a \in A$ and $b \in B$ with $g(b)= a$ we have
$$
g(h(a))= g(f(b)) = g(b) = a.
$$
We have just proved that $\alg A$ is a retract of $\alg B$; by Proposition \ref{prop:retractproj1} the conclusion follows.
\end{proof}

In particular  any finite Brouwerian semilattice (i.e., an idempotent hoop) is finitely projective in
the variety of Brouwerian semilattices. Since the latter is locally finite \cite{Kohler1981} we get at once that all the finitely presented Brouwerian semilattices are projective (a fact already noted in \cite{Ghilardi1997}).

The class of $\imp$-subreducts of the variety of commutative and integral residuated lattice is the (proper) quasivariety $\mathsf{BCK}$ of BCK-algebras; those algebras were introduced by K. Iseki \cite{Iseki1962} and they have been widely studied since; the fact that they are indeed subreducts of commutative and integral residuated lattices is again a consequence of the quoted result in \cite{OnoKomori1985}.  In \cite{Ferr2000} the class of BCK-algebras that are implicative subreducts of hoops has been described: it is a variety called $\mathsf{HBCK}$ and the algebras therein are called HBCK-algebras.
An analogous of Theorem \ref{mth} follows with the same proof in the reduced language.

\begin{theorem}\label{thm:hbck} Every finite HBCK-algebra  is finitely projective in the variety of HBCK-algebras,  hence in any locally finite variety of HBCK-algebras  every finitely presented (equivalently, finite) algebra is projective.
 \end{theorem}

In order to transfer our argument to full hoops or to $\mathsf{HL}$-algebras we have to take care of the join. The easiest case is the one in which the join is definable, which is equivalent to saying that the hoop satisfies the equation (PJ) (see \cite{Agliano2018c}).  Thus, prelinear hoops are full hoops (i.e., they are residuated lattices) and the conclusion applies.

\begin{corollary}\label{cor:projlfbasichoops}
	Every finite basic hoop is finitely projective in the variety of basic hoops, hence in any locally finite variety of basic hoops every finitely presented (equivalently, finite) algebra is projective.
\end{corollary}

If we remove the hypothesis of being locally finite, the previous result does not hold. Indeed, for instance, not all finitely presented Wajsberg hoops are projective, as shown in \cite{Ugolini2022}. However, there is another interesting variety of hoops that is not locally finite, but for which the same property holds: the variety of cancellative hoops. {\em Cancellative hoops} are basic hoops satisfying the cancellativity law:
$$ x \to (x y) = y . \qquad (canc)$$
Cancellative hoops can be seen as negative cones of abelian lattice ordered abelian groups (or {\em $\ell$-groups} for short), and actually, they are categorically equivalent to $\ell$-groups \cite{DiNolaLettieri1994}. Now, projective and finitely generated $\ell$-groups coincide with finitely presented $\ell$-groups, as shown in \cite{Beynon1977}.
Just as the properties of being projective and finitely presented, in every variety also the concept of an algebra being finitely generated is categorical, i.e., it can be described in the abstract categorical setting by properties of morphisms (see Theorem 3.11 and 3.12 in [1]), and thus all these notions are preserved by categorical equivalences. Therefore, we have the following result.

\begin{proposition}\label{prop:cancproj}
Finitely presented cancellative hoops are exactly the finitely generated and projective in their variety.
\end{proposition}

Are there non prelinear varieties of commutative and integral residuated lattices for which an analogous of Theorem \ref{mth} holds? We do not know; there are results in the literature (for instance Lemma 8.2 in \cite{OlsonRafVanAlten2008}) that suggest that one would need to use different proof techniques than the ones considered here.

If we add the constant $0$ to the signature to represent the least element of the considered structures, the argument used above does not work. Indeed, in Theorem \ref{mth}, given any surjective homomorphism to a finite algebra, we define an embedding that testifies the retraction which is not necessarily preserving the lower bound.
 However, we can prove the following (weaker) result.

\begin{theorem}\label{cor:projlfBL}
	Every finite bounded hoop  is finitely zero-projective in the variety of bounded hoops; hence in  any locally finite variety of bounded hoops every finitely presented algebra is zero-projective.
\end{theorem}
\begin{proof}
	Let $\alg A,\alg B$ be finite bounded hoops and let $g\colon \alg B \longrightarrow \alg A$ be a surjective homomorphism such that $g^{-1}(0) =0$. Then, by the same argument as Theorem \ref{mth}, we can find an endomorphism $f$ of the $0$-free reduct of $\alg B$ such that $\alg A$ and $f(\alg B)$ are isomorphic through $h$ as $0$-free structures. It only remains to be checked that the map $h$ is a homomorphism in the full signature. Since we showed that $gh$ is the identity, $g(h(0)) = 0$, which implies that $h(0) = 0$ since $g^{-1}(0) = 0$.
Hence the conclusion follows.
\end{proof}

Prelinear bounded hoops are (termwise equivalent to) BL-algebras so we get:

\begin{corollary} Every finite BL-algebra  is finitely zero-projective in the variety of BL-algebras; hence in  any locally finite variety of BL-algebras every finitely presented algebra is zero-projective.
\end{corollary}

  Since any projective algebra is zero-projective, it is sensible to ask if it possible to characterize the projective algebras  in a different way. We have seen that $\alg 2$ is projective in every subvariety of $\mathsf{FL_{ew}}$ (Lemma \ref{lemma:2proj}), since it is the free algebra over the empty set of generators. Moreover, the following holds.

\begin{lemma}\label{lemma:2retract}
	$\alg 2$ is a retract of every free algebra in every subvariety of $\mathsf{FL_{ew}}$.
\end{lemma}
\begin{proof}
$\alg 2$ is (isomorphic to) a subalgebra of every free algebra $\alg F$ in every subvariety of $\mathsf{FL_{ew}}$.
	The retraction is then testified by the inclusion map and any (necessarily surjective) homomorphism that maps the generators of $\alg F$ to $\{0,1\}$.
\end{proof}
From this fact we derive another property of projective $\mathsf{FL_{ew}}$-algebras.

\begin{proposition}\label{lemma:2homimage}
Let $\vv V$ be a variety of $\mathsf{FL_{ew}}$-algebras. If $\alg A$ is projective in $\vv V$, then $\alg A$ has $\alg 2$ as an homomorphic image.
\end{proposition}

Now, any algebra that has a negation fixpoint, e.g., the totally ordered three-element Wajsberg algebra $\alg{\L}_3$, cannot have $\mathbf 2$ as homomorphic image. Thus, there are finite bounded residuated (semi)lattices that are zero-projective but not projective. However, the necessary condition turns out to be also sufficient for
bounded hoops, and hence for $\mathsf{BL}$-algebras.  We start with a technical lemma.

\begin{lemma}\label{lemma:semantical} The variety of bounded hoops satisfies the quasiequation
\begin{equation}
x^2 \app x \quad\Longrightarrow\quad  (x \imp y) \imp \neg\neg x \app \neg\neg x.\tag{Q}
\end{equation}
\end{lemma}
\begin{proof} It is well-known that the variety of bounded hoops has the FEP \cite{BlokFerr2000} and therefore it is generated as a quasivariety by its finite algebras; since
any finite algebra is a subdirect product of finite subdirectly irreducible algebras it is enough to prove (Q) for all the finite subdirectly irreducible bounded hoops.
Now \cite{BlokFerr2000} any finite subdirectly irreducible bounded hoop $\alg A$ is an ordinal sum $\alg F \oplus \alg S$ where $\alg S$ is a finite totally ordered Wajsberg hoop
and $\alg F$ is a (possibly trivial) bounded hoop. 

We observe that in a finite totally ordered Wajsberg hoop the only idempotents are $1$ and the bottom element so (Q) is satisfied.
We will prove the thesis by induction on the size of $\alg A$; first the only two-element bounded hoop is $\mathbf 2$ and the only three-element subdirectly irreducible bounded hoop are
$\mathbf 2 \oplus \mathbf 2$ and the three-element Wajsberg chain. 
 (Q) holds in both these structures by inspection. Now suppose we have proved the statement for all subdirectly irreducible algebras
of size $\le n$ and let $\alg A = \alg F \oplus \alg S$ with $|A|=n+1$. If $\alg F$ is trivial , then $\alg A \cong \alg S$ and (Q) holds; otherwise $|F|\le n$ and $\alg F$ is a subdirect product of subdirectly
irreducible algebras of size $\le n$ for which (Q) holds by induction hypothesis. Hence (Q) holds in $\alg F$ as well.  Finally, if $x \in S, x^2 = x,$  and $y \in F$ then $\neg\neg x =1$ and
$$
(x \imp y) \imp \neg\neg x = y \imp 1 =1 = \neg\neg x;
$$
if $x \in F, x^2 = x,$ and $y \in S$, then
$$
(x \imp y) \imp \neg\neg x = 1 \imp \neg\neg x = \neg\neg x.
$$
This concludes the induction and proves the thesis.
\end{proof}

If $\alg A$ is any bounded hoop, then the double negation is a $\cdot$-homomorphism, that is, for all $a,b \in A$
$$
\neg\neg(a\cdot b) = \neg\neg a \cdot \neg\neg b.
$$
This has been shown by Cignoli and Torrens in \cite[Theorem 4.8]{CT04}.
It follows that if $u \in A$ is idempotent, so is $\neg\neg u$.

The following lemma was stated in \cite{Dzik2008} for $\mathsf{BL}$-algebras without proof. While we were not able to derive it syntactically, using Lemma \ref{lemma:semantical} we show that the conclusions hold for bounded hoops (i.e., prelinearity and the join are not needed).

\begin{lemma}\label{lemma:dzik} If $\alg A$ is  a bounded hoop, $a,b,u \in \alg A$ and $u$ is idempotent, then
\begin{align*}
ua = u \meet a \tag{U1}\\
u \imp ua = u\imp a \tag{U2}\\
\neg\neg u((u \imp a) \imp (u \imp b))&= (u \imp a) \imp \neg\neg u(u \imp b)\tag{U3}\\
(u \imp a) \imp (u \imp b) &= \neg\neg u(u \imp a) \imp \neg\neg u (u \imp b) \\
&=  \neg\neg u(u \imp a) \imp (u \imp b).\tag{U4}
\end{align*}
\end{lemma}
\begin{proof}  Observe that $ua \le u \meet a$ by integrality; conversely
$$
u \meet a = u (u \imp a) = u^2(u \imp a) = u (u\meet a) \le ua.
$$
Thus (U1) holds, and (U2) is an easy consequence: $u (u \to a) = u \land a = ua$ implies $u \to a \leq u \to ua$, and the other inequality is a consequence of residuation.

Now by Lemma \ref{lemma:semantical} if $u$ is idempotent and $a \in A$ then
$$
(u \imp a) \imp \neg\neg u = \neg \neg u.
$$
Using (U1) and the fact that in residuated structures the implication distributes over meet we get the following:
\begin{align*}
(u \imp a) \imp \neg\neg u(u \imp b)&= (u \imp a) \imp[\neg\neg u \meet (u \imp b)]\\
&= [(u \imp a) \imp \neg\neg u] \meet [(u \imp a) \imp (u\imp b)]\\
&= \neg \neg u \meet [(u \imp a) \imp (u\imp b)] \\
&= \neg\neg u[(u \imp a) \imp (u\imp b)]
\end{align*}
and (U3) holds.
For (U4), the first equality is shown as follows:
\begin{align*}
\neg\neg u(u \imp a) \imp (u \imp b) &\le u(u \imp a) \imp (u \imp b) \\
&=u \imp [(u \imp a) \imp (u \imp b)] \\
& = u \imp (u \imp (a \imp b)) \\
&= u \imp (a \imp b) \\
&= (u \imp a) \imp (u \imp b)\\
&\le \neg\neg u( u \imp a) \imp (u \imp b).
\end{align*}
Finally,
\begin{align*}
\neg\neg u(u \imp a) \imp \neg\neg u(u \imp b) &= \neg\neg u(u \imp a) \imp [\neg\neg u \meet(u \imp b)]\\
& [\neg\neg u(u \imp a) \imp \neg\neg u] \meet [\neg\neg (u \imp a) \imp (u \imp b)] \\
&=\neg\neg u (u \imp a) \imp (u \imp b).
\end{align*}
Hence (U4) holds as well.
\end{proof}

\begin{lemma} \label{dzik}Let $\alg A$ be a bounded hoop, and let $\alg 2$ denote its subalgebra with domain $\{0,1\}$. If $\f\colon \alg A \longrightarrow \mathbf 2$ is a homomorphism and $u \in A$ is idempotent, then the map
$$
f(x) = (u \imp x) (\neg u \imp \f(x))
$$
is an endomorphism of $\alg A$.
\end{lemma}
\begin{proof} We first prove that $f$ is a $\cdot$-homomorphism. Let $x,y \in A$; if $\f(xy) = 1$, then $\f(x) =\f(y) =1$ and
\begin{align*}
f(xy) &= (u \imp xy)(\neg u \imp 1) = (u \imp x)(u \imp y)\\
&= (u \imp x)(\neg u \imp 1)(u \imp y) (\neg u \imp 1) = f(x)f(y).
\end{align*}
Suppose that $\f(xy) = 0$, $\f(x)= 1$ and $\f(y)=0$; then
\begin{align*}
f(xy) &= (u \imp xy)\neg\neg u = (u \imp x)(u \imp y)\neg\neg u\\
&= (u \imp x)(\neg u \imp 1)(u \imp y) \neg\neg u = f(x)f(y).
\end{align*}
Finally suppose that $\f(xy) = \f(x) = \f(y)=0$; then
\begin{align*}
f(xy) &= (u \imp xy)\neg\neg u = (u \imp x)(u \imp y)\neg\neg u\\
&= (u \imp x)\neg\neg u(u \imp y) \neg\neg u = f(x)f(y).
\end{align*}

Thus, $f$ is a $\cdot$-homomorphism. 
%Suppose first that $\f(x \to y) = 0$, thus $\f(x) = 1, \f(y) = 0$. We get: 
%\begin{align*}
%f(x \to y) &= (u \imp (x \to y))\neg\neg u = \neg\neg u((u \to x) \to (u \to y))\\
%&= (u \to x) \to ((u \to y)\neg\neg u)= f(x)\to f(y).
%\end{align*}
% Suppose now $\f(x \to y) = 1$, then
%$$f(x \to y) = u \to (x \to y) = (u \to x) \to (u \to y).$$
%We have three cases to check: $\f(x) = 0, \f(y) =1$; $\f(x) = \f(y) = 0$; $\f(x) = \f(y) = 1$.
%If $\f(x) = 0, \f(y) =1$, \begin{align*}
% 	f(x) \to f(y) = ((u \to x) \neg\neg u) \to (u\to y) = (u \to x) \to (u \to y)
% \end{align*}
% by (U4).
%If $\f(x) = \f(y) = 0$, again using (U4)\begin{align*}
% 	f(x) \to f(y) = ((u \to x) \neg\neg u) \to ((u\to y)\neg\neg u) = (u \to x) \to (u \to y).
% \end{align*}
%Finally, if $\f(x) = \f(y) = 1$, $f(x) \to f(y) = (u \to x) \to (u \to y)$ and therefore $f$ preserves $\to$.
%Since $f(1) = (u \to 1) (\neg u \to 1) = 1$ and $f(0) = \neg u \cdot \neg\neg u = 0$, $f$ is an endomorphism of $\alg A$.
%\end{proof}
%
%
%
%\begin{lemma} \label{dzik}(see \cite{Dzik2008}, Theorem 3.8) Let $\alg A$ be a bounded hoop and suppose that there exists a homomorphism $\f: \alg A \longrightarrow \mathbf 2$; then for any idempotent $u \in A$ the function
%$$
%f(x) := (u \imp x)(\neg u \imp \f(x))
%$$
%is an endomorphism of $\alg A$.
%\end{lemma}
%\begin{proof} First observe that is is enough to prove that $f$ is a $\{\imp,\cdot,0,1\}$-endomorphism   since $\meet$ is definable by divisibility. That it is a $\cdot$-endomorphism is Lemma \ref{lemma:cdotend} while 
We now show that $f$ preserves $\imp$ using Lemma \ref{lemma:dzik}; consider:
\begin{align*}
f(a \imp b) &= (u \imp (a \imp b))(\neg u \imp \f(a \imp b)))\\
&= [(u \imp a) \imp (u\imp b)][\neg u \imp (\f(a) \imp \f(b))].
\end{align*}
Suppose now that $\f(a)=1$ and $\f(b)=0$; then $\f(a) \imp \f(b) =0$ and the conclusion follows from (U3). In any other case $\f(a) \imp \f(b) =1$ and the conclusion is either obvious or  follows from (U4).
Finally $f(1) =1$ and
$$
f(0) = \neg u \cdot (\neg u \imp 0) = \neg u \meet 0 = 0.
$$
This completes the proof.
\end{proof}
We are now ready to prove the following characterization result. 
\begin{theorem}\label{mth0} A finite bounded hoop $\alg A$ is finitely projective in the variety of bounded hoops if and only if there exists a homomorphism $\f\colon \alg A \longrightarrow \mathbf 2$.
\end{theorem}

\begin{proof} We have already observed that the condition is necessary, we now show that it is sufficient. Let $\alg A$ be a finite bounded hoop, let $\f\colon \alg A \longrightarrow \mathbf 2$ be a homomorphism and let $\alg B$ be a finite bounded hoop such that
$g\colon\alg B \longrightarrow \alg A$ is a surjective homomorphism. Then $\f g(x)$ is a homomorphism from $\alg B $ to $\mathbf 2$. Let $\f'$ the corresponding homomorphism from $\alg B$ to $\mathbf 2 \sse \alg B$. As in the proof of Theorem \ref{mth}, if $\th = \op{ker}(g)$, then  $F = 1/\th$ is a filter of $\alg B$; since $\alg B$ is finite $F$ must have a minimum $u$ which is necessarily idempotent. Observe also that $\f'(u)=1$; in fact $u \in F$ which means that $g(u) =g(1)$ and thus
 $$
 \f'(u) = \f g(u) = \f g(1) = \f(1) = 1.
 $$
 Let $f$ be the endomorphism of $\alg B$ (relative to $u$ and $\f'$) resulting from Lemma \ref{dzik}; then $f(b) = 1$ implies $u \imp b =1$ and thus $b \in F$. Conversely if $b \in F$ then
$$
f(b) = (u \imp b)(\neg u \imp \f'(b))\mathrel{\th} (u \imp 1)(\neg u \imp \f'(1)) = 1.
$$ 
Moreover, we can show that $f$ is idempotent. Indeed, if $\f'(x) = 0$, 
\begin{align*}
f(f(x)) &= (u \to \neg\neg u(u \to x)) \neg\neg u = [(u \to \neg\neg u) \land (u \to (u \to x))]\neg\neg u \\
&= (u \to x) \neg\neg u = f(x),
\end{align*}
while if $\f'(x) = 1$, 
\begin{align*}
f(f(x)) &= (u \to (u \to x)) = u \to x = f(x).
\end{align*}
Now we can argue as in Theorem \ref{mth} to conclude that $\alg A$ is a retract of $\alg B$ and the conclusion follows again from  Proposition \ref{prop:retractproj1}.
\end{proof}

\begin{corollary}\label{cor:boundedhoops}Let $\vv V$ be a locally finite variety of bounded hoops; then a finite $\alg A \in \vv V$ is projective in $\vv V$ if and only if there is a homomorphism $\f\colon \alg A \longrightarrow \mathbf 2$.
\end{corollary}

The case of BL-algebras is managed in the usual way; since the join is definable by prelinearity all the arguments go through without change.

\begin{corollary}\label{corollary:projectiveBL}Let $\vv V$ be a locally finite variety of BL-algebras; then a finite $\alg A \in \vv V$ is projective in $\vv V$ if and only if there is a homomorphism $\f\colon \alg A \longrightarrow \mathbf 2$.
\end{corollary}

Since every finite algebra $\alg A$ in a variety is a subdirect product of finite subdirectly irreducible algebras, and all the subdirect factors are homomorphic images of $\alg A$, we can sharpen our results some more.

\begin{theorem}\label{theorem:projsubirr} Let $\vv V$ be a locally finite variety of bounded hoops or $\mathsf{BL}$-algebras such that every finite subdirectly irreducible in $\vv V$ has $\mathbf 2$ as homomorphic image. Then every finitely presented
algebra in $\vv V$ is projective.
\end{theorem}

The theorem above is (yet another) reason why every finitely presented (i.e., finite) Boolean algebra is projective: the variety of Boolean algebras is locally finite and the only subdirectly irreducible is $\mathbf 2$. A more intriguing example is the following: a variety of $\mathsf{FL_{ew}}$-algebras is {\em Stonean} if it satisfies the equation
$$
\neg x \join \neg\neg x \app 1.
$$
It is the straightforward consequence of  the characterization of the subdirectly irreducible $\mathsf{BL}$-algebras in \cite{AglianoMontagna2003} that a finite subdirectly irreducible algebra in a Stonean variety is
of the form  $\alg A \cong \mathbf 2 \oplus \alg B$, where $\alg B$ is a totally ordered hoop. Since $\alg B$ is a filter of $\alg A$, we can collapse it and get $\mathbf 2$ as a homomorphic image of $\alg A$. Hence, locally finite Stonean varieties of $\mathsf {BL}$-algebras fall under the scope of Theorem \ref{theorem:projsubirr}. Stonean $\mathsf{BL}$-algebras are a particular instance of varieties of residuated lattices with a so-called {\em Boolean retraction term} \cite{CT12}; these varieties will be studied in depth in a forthcoming paper \cite{AU}.

\section{Algebraic unification theory: an application}\label{section:unif}

The origin of unification theory is usually attributed to Julia Robinson \cite{Robinson1965}.
The classical syntactic unification problem is as follows:  given two term $s,t$ (built from function symbols and variables), find a {\em unifier} for them, that is, a uniform replacement of the variables occurring in $s$ and $t$ by other terms that makes $s$ and $t$ identical.  When the latter syntactical identity is replaced by equality modulo a given equational theory $E$, one speaks of $E$-unification.  Unsurprisingly, $E$-unification can be considerably harder than syntactic unification, even when the theory $E$ is fairly well understood.  To ease this problem S. Ghilardi in \cite{Ghilardi1997} proposed a new (equivalent) approach and it is the approach we are going to follow here.  A unification problem
for a variety $\vv V$ is a finitely presented algebra $\alg A$ in $\vv V$; a {\em solution} is a homomorphism $u\colon\alg A \longrightarrow \alg P$, where $\alg P$ is a projective algebra in $\vv V$. In this case $u$ is called a {\em unifier} for $\alg A$ and we say that $\alg A$ is {\em unifiable}. If $u_1, u_2$ are two different unifiers for an algebra $\alg A$ (with projective targets $\alg P_1$ and $\alg P_2$) we say that $u_1$ is {\em more general} than $u_2$ if
there exists a homomorphism $m\colon \alg P_1 \longmapsto \alg P_2$ such that $mu_1 = u_2$. The relation  ``being less general of'' is a preordering on the unifiers of $\alg A$, thus we can consider the associated equivalence relation; then the equivalence classes (i.e., the unifiers that are ``equally general'') form a partially ordered set $U_\alg A$.  It is customary to assign a type (unitary, finitary, infinitary, or nullary) to the algebra according to how many maximal elements $U_\alg A$ has; we are particularly interested in the case of {\em unitary type} which happens if $U_\alg A$ has a maximum, i.e. there is only one maximal element which plays the role of a ``best solution'' to the unification problem, since all other ones can be obtained starting from it, by further substitutions. The type of a variety $\vv V$ is unitary if every unifiable algebra $\alg A \in \vv V$ has unitary type. In this case where the type of $\alg A$ is unitary then the maximum in $U_\alg A$ is called the {\em most general unifier} ({\em mgu} for short) of $\alg A$. We separate the case in which the {\em mgu} is of a special kind: we say that $\alg A$ has {\em strong unitary type} if its {\em mgu} is the identity. A variety $\vv V$ has  strong unitary type if every unifiable algebra in $\vv V$ has strong unitary type.
 The connection with the present paper is given by the following observation.

\begin{proposition}\label{prop:stunitary}
\label{unitary}  Let $\vv V$ be any variety; then the following are equivalent.
 \begin{enumerate}
 \item $\vv V$ has strong unitary unification type;
 \item for any finitely presented algebra $\alg A \in \vv V$, $\alg A$ is unifiable if and only if it is projective.
 \end{enumerate}
\end{proposition}

In \cite{Ghilardi1997} S. Ghilardi used implicitly this fact to prove that any variety of Brouwerian semilattices, and any variety generated by a quasi-primal algebra, have strong unitary type. Any finitely presented (i.e. finite)  Brouwerian semilattice is unifiable since it a has homomorphism into $\{1\}$ that is the free algebra on the empty set; on the other hand, any finite Brouwerian semilattice is projective by Theorem \ref{mth}. The same holds (see the first section of this paper)   for varieties generated by a quasi-primal algebras with no minimal nontrivial subalgebras (e.g., the variety of Boolean algebras). In case the variety is generated by a quasi-primal algebra with nontrivial minimal subalgebras, then it is not true that every finite algebra is projective, but still every unifiable algebra is projective \cite{Ghilardi1997}.

In this paper we have showed that several interesting varieties are such that all finitely presented algebras are projective, therefore showing that point (2) of Proposition \ref{unitary} holds. Moreover, in the  case of bounded commutative residuated lattices (i.e. $\mathsf{FL_{ew}}$-algebras) we can rephrase Proposition \ref{prop:stunitary} in a more transparent way.
In varieties of $\mathsf{FL_{ew}}$-algebras any unifiable algebra must have a surjective homomorphism on the two element algebra $\mathbf 2$ (Proposition \ref{lemma:2homimage}), and since $\alg 2$ is projective in every variety of $\mathsf{FL_{ew}}$-algebras (Lemma \ref{lemma:2retract}) we get:

\begin{proposition}\label{prop:zerostunitary}
	\label{ground} For a variety $\vv V$ of $\mathsf{FL_{ew}}$ the following are equivalent:
\begin{enumerate}
\item $\vv V$ has strong unitary type;
\item for any finitely presented $\alg A \in \vv V$, $\alg A$ has $\mathbf 2$ as a homomorphic image if and only if $\alg A$ is projective.
\end{enumerate}
\end{proposition}

The property (2) in Lemma \ref{ground} has also been called {\em groundedness}. D. Valota et. al. (private communication) classified several subvarieties of $\mathsf{MTL}$ enjoying groundedness; we remark that in some cases their results overlap with ours.
We get the following conclusions. 
\begin{theorem} The following varieties, and their corresponding logics, have strong unitary unification type:
\begin{enumerate}
\item all locally finite subvarieties of hoops and of HBCK-algebras;
\item all locally finite subvarieties of bounded hoops and BL-algebras;
\item cancellative hoops.
%\item locally finite varieties of Stonean $\mathsf{BL}$-algebras, hence Boolean algebras.
\end{enumerate}
\end{theorem}
\begin{proof}
	The proof follows from Propositions \ref{prop:stunitary} and \ref{prop:zerostunitary} together with, respectively: Theorems \ref{mth} and \ref{thm:hbck}; Corollaries \ref{cor:boundedhoops} and \ref{corollary:projectiveBL}; Proposition \ref{prop:cancproj}.% Theorem \ref{theorem:projsubirr}.
\end{proof}
In particular then, all locally finite varieties of MV-algebras have strong unitary unification type, while the variety of all MV-algebras does not, in fact its unification type is nullary, the worst-case scenario (\cite{MarraSpada2013}).  Projective algebras in locally finite varieties of MV-algebras have been characterized in \cite{DiNolaGrigoliaLettieri2008}.
\begin{theorem}\label{dinola}\cite{DiNolaGrigoliaLettieri2008} Let $\vv V$ be a locally finite variety of MV-algebras; then a finite algebra $\alg A \in \vv V$ is projective if and only if it is isomorphic with $\mathbf 2 \times \alg A'$ for some finite $\alg A' \in \vv V$.
\end{theorem}

Using this characterization, we can get a simple alternative proof of the fact that locally finite varieties of MV-algebras have strong unitary unification type.
\begin{proposition}
	Every locally finite variety $\vv V$ of MV-algebras has strong unitary type.
\end{proposition}
\begin{proof}
If $\alg A$ is finite and unifiable then there is a homomorphism
$f\colon\alg A \longrightarrow \mathbf 2 \times \alg B$ for some finite algebra $\alg B \in \vv V$. Let $p_1$ be the first projection on the direct product and let $g=p_1f$; then the function
$$
h(a) = \la g(a), a\ra
$$
is a monomorphism of $\alg A$ into $\mathbf 2 \times \alg A$, that is projective by Theorem \ref{dinola}. 
Let us consider the second projection $p_2$ on $\alg 2 \times \alg A$. Then $p_2h = id_{\alg A}$ thus
 $\alg A$ is a retract of a projective algebra and therefore projective as well. Notice that if we project this monomorphism
on the first coordinate we get a homomorphism from $\alg A$ to $\mathbf 2$ as predicted by Proposition \ref{prop:zerostunitary}.
\end{proof}

\section{Conclusions}

In this paper  we have investigated commutative integral residuated lattices mainly using some well established constructions, such as the ordinal sum, obtaining results on varieties with some strong structural properties, such as being divisible. In particular, we have shown interesting varieties whose finitely presented algebras are all projective.
This by no means exhausts the investigation of projectivity in commutative residuated lattices and $\mathsf{FL_{ew}}$-algebras. Here we will simply give a preview of what we plan to do next.

The topic of residuated commutative lattices with a Boolean retraction term is promising; there are some general results, which will appear in \cite{AU}.
Moreover in \cite{Ghilardi1999} S. Ghilardi investigated unification in Heyting algebras, proving that they have finitary (but not unitary) type;
in particular he proved that the variety of Stonean Heyting algebras has unitary type (and it can be deduced that it is not strong), and that it is the largest subvariety of Heyting algebras whose type is unitary.
The reader should keep in mind that such results in \cite{Ghilardi1999} are obtained via syntactical methods; though the symbolic and algebraic setting are equivalent, the translation of logical proofs into algebraic ones is anything but straightforward. As a matter of fact we believe that a careful algebraic analysis of the proofs could shed some light on unification in Stonean residuated lattices, and possibly even on commutative residuated lattices with a Boolean retraction term.

We intend to address these problems (and more) in the immediate future.
\section*{Funding}
This work has received funding from the European Union's Horizon 2020 research and innovation programme under the Marie Sklodowska-Curie grant agreement No 890616 awarded to S. Ugolini.

%%%%%%%%%%%%%%%%%%%%%%%%%%%%%%%%%%%%%%%%%%%%%%%%%%%%%%%%%%%%%%%%%%%%%%
%% Acknowledgments (Optional)
%%%%%%%%%%%%%%%%%%%%%%%%%%%%%%%%%%%%%%%%%%%%%%%%%%%%%%%%%%%%%%%%%%%%%%

%%%%%%%%%%%%%%%%%%%%%%%%%%%%%%%%%%%%%%%%%%%%%%%%%%%%%%%%%%%%%%%%%%%%%%
%% BIBLIOGRAPHY
%%%%%%%%%%%%%%%%%%%%%%%%%%%%%%%%%%%%%%%%%%%%%%%%%%%%%%%%%%%%%%%%%%%%%%

%%%%%%%%%%%%%%%%%%%%%%%%%%%%%%%%%%%%%%%%%%%%%%%%%%%%%%%%%%%%%%
\end{document}